\documentclass[11pt,reqno]{amsart}

\usepackage[all]{xy}
\usepackage{url}
\usepackage{amssymb}
\usepackage{hyperref}
\usepackage[margin=1.2 in]{geometry}
\usepackage{color,soul}
\usepackage{amsthm}
\usepackage{lipsum}
\usepackage{enumitem}
\usepackage{enumitem}
\setbox1=\hbox{$\bullet$}

\usepackage{mathtools}

\DeclareMathOperator{\td}{tr.deg}

\DeclareMathOperator{\Frac}{Frac}
\DeclareMathOperator{\id}{id}

\theoremstyle{plain}
\newtheorem{theorem}[subsubsection]{Theorem}
\newtheorem{proposition}[subsubsection]{Proposition}
\newtheorem{lemma}[subsubsection]{Lemma}

\newtheorem{corollary}[subsubsection]{Corollary}
\newtheorem{claim}[subsubsection]{Claim}

\theoremstyle{definition}
\newtheorem{definition}[subsubsection]{Definition}
\newtheorem{example}[subsubsection]{Example}
\newtheorem{examples}[subsubsection]{Examples}

\newtheorem{remark}[subsubsection]{Remark}
\newtheorem{remarks}[subsubsection]{Remarks}

\theoremstyle{definition}

\theoremstyle{plain}
\newtheorem{innercustomthm}{Theorem}
\newenvironment{customthm}[1]
  {\renewcommand\theinnercustomthm{#1}\innercustomthm}
  {\endinnercustomthm}

\theoremstyle{plain}

\theoremstyle{plain}

\theoremstyle{definition}

\theoremstyle{definition}

\numberwithin{equation}{subsubsection}
\numberwithin{equation}{subsubsection}
\normalfont
\usepackage[T1]{fontenc}{}
\usepackage{setspace}

\usepackage{tikz-cd}

\makeatletter
\def\@tocline#1#2#3#4#5#6#7{\relax
  \ifnum #1>\c@tocdepth 
  \else
    \par \addpenalty\@secpenalty\addvspace{#2}%
    \begingroup \hyphenpenalty\@M
    \@ifempty{#4}{%
      \@tempdima\csname r@tocindent\number#1\endcsname\relax
    }{%
      \@tempdima#4\relax
    }%
    \parindent\z@ \leftskip#3\relax \advance\leftskip\@tempdima\relax
    \rightskip\@pnumwidth plus4em \parfillskip-\@pnumwidth
    #5\leavevmode\hskip-\@tempdima
      \ifcase #1
       \or\or \hskip 1em \or \hskip 2em \else \hskip 3em \fi%
      #6\nobreak\relax
    \dotfill\hbox to\@pnumwidth{\@tocpagenum{#7}}\par
    \nobreak
    \endgroup
  \fi}
\makeatother
\usepackage{hyperref}

\title{Frobenius splitting of valuation rings and $F$-singularities of centers}
\author{Rankeya Datta\vspace{-2ex}}
\address{851 S Morgan St, Science and Engineering Offices, Office 411, Chicago, IL 60607}
\email{rankeya@uic.edu}
\thanks{}

\begin{document}

\maketitle
\vspace{-4ex}

\begin{abstract}
Using a local monomialization result of Knaf and Kuhlmann, we prove that the valuation ring of an Abhyankar valuation of a function field over a perfect ground field of prime characteristic is Frobenius split. We show that a Frobenius splitting of a sufficiently well-behaved center lifts to a Frobenius splitting of the valuation ring. We also investigate properties of valuations centered on arbitrary Noetherian domains of prime characteristic. In contrast to \cite{DS16, DS17}, this paper emphasizes the role of centers in controlling properties of valuation rings in prime characteristic. 
\end{abstract}

\section{Introduction}
One of the goals of this paper is to prove the following result.

\begin{customthm}{A}
\label{Theorem-A}
Let $K$ be a finitely generated field extension of a perfect field $k$ of prime characteristic. Then the valuation ring of any Abhyankar valuation of $K/k$ is Frobenius split.
\end{customthm}

\vspace{-1ex}

\noindent For a valuation $\nu$ of $K/k$, if $\Gamma_{\nu}$ is the value group and $(R_\nu, \mathfrak{m}_\nu, \kappa_{\nu})$ the valuation ring, then $\nu$ is \emph{Abhyankar} if
$$\dim_{\mathbb{Q}}(\mathbb{Q} \otimes_{\mathbb Z} \Gamma_{\nu}) + \td{\kappa_{\nu}/k} = \td{K/k}.$$
Abhyankar valuations, often also called \emph{quasimonomial} valuations in characteristic $0$, extend the class of valuations associated to prime divisors on normal models, since $\dim_{\mathbb Q} (\mathbb{Q} \otimes_{\mathbb Z} \Gamma_{\nu}) = 1$ and $\td{\kappa_{\nu}/k} = \td{K/k} - 1$ for divisorial valuations. Nevertheless, non-divisorial Abhyankar valuations arise naturally in geometry (see \cite{Spi90, ELS03, JF04, FJ05, JM12, Mus12, Tem13, RS14, Tei14, Pay14, Blu18} for some applications), and they possess  many of the good properties of their Noetherian counterparts. For example, the value group of any Abhyankar valuation is a free abelian group of finite rank, and its residue field is a finitely generated extension of the ground field. 

Divisorial valuation rings over perfect fields of prime characteristic are Frobenius split. Indeed, when a Noetherian local ring $R$ is \emph{$F$-finite}, that is when the Frobenius ($p$-th power) endomorphism $F: R \rightarrow R$ is a finite ring map, a famous result of Kunz shows that regularity of $R$ is characterized by $R$ being free over its $p$-th power subring $R^p$ \cite[Theorem 2.1]{Kun69}. Therefore $F$-finite regular rings, and consequently divisorial valuation rings over perfect fields, are Frobenius split (the same conclusion can be drawn using the Direct Summand Theorem). Thus Theorem \ref{Theorem-A} extends a well-known fact about divisorial valuation rings to a class of valuation rings that behaves the most like divisorial ones.

A key ingredient in our proof of Theorem \ref{Theorem-A} is the following result of Knaf and Kuhlmann which says that one can locally monomialize finite subsets of Abhyankar valuation rings in \emph{any} characteristic. 

\begin{customthm}{{1}}{\cite[Theorem 1]{KK05}}
\label{local-uniformization}
Let $K$ be a finitely generated field extension of a field $k$ of any characteristic, and let $\nu$ be an Abhyankar valuation of $K/k$ such that the residue field $\kappa_{\nu}$ is separable over $k$. Then for any finite set $Z \subseteq R_{\nu}$, $R_{\nu}$ is centered on a regular local ring $(A, \mathfrak{m}_A, \kappa_A)$ essentially of finite type over $k$ with fraction field $K$ satisfying the following properties:
\begin{enumerate}
\item The Krull dimension of $A$ equals $d \coloneqq \dim_{\mathbb{Q}}(\mathbb{Q} \otimes_{\mathbb Z} \Gamma_{\nu})$.

\item $Z \subseteq A$, and there exists a regular system of parameters $\{x_1, \dots, x_d\}$ of $A$ such that every $z \in Z$ admits a factorization
$$z = ux^{a_1}_1\dots x^{a_d}_d,$$
for some $u \in A^{\times}$, and $a_i \in \mathbb{N} \cup \{0\}$.
\end{enumerate}
\end{customthm}

\noindent When the ground field $k$ is perfect, any Abhyankar valuation $\nu$ admits a local monomialization in the sense of  Theorem \ref{local-uniformization} because $\kappa_\nu$ is automatically separable over $k$. Theorem \ref{local-uniformization} allows us to additionally choose a center of  $\nu$ on a regular local $k$-algebra $A$ such that $A$ has a regular system of parameters $\{x_1,\dots,x_d\}$ whose valuations freely generate $\Gamma_\nu$ and the residue field of $A$ coincides with $\kappa_{\nu}$ (Lemma \ref{regular-center}). Our strategy then is to identify a suitable Frobenius splitting of $A$ that lifts to a Frobenius splitting of $R_{\nu}$. 

A proof of Theorem \ref{Theorem-A} was announced in \cite[Theorem 5.1]{DS16}, where Frobenius splitting was deduced as a consequence of the \emph{incorrect} assertion that Abhyankar valuation rings are F-finite. On the contrary, \cite[Theorem 0.1]{DS17} shows that finiteness of Frobenius for valuation rings of function fields (a \emph{function field} is a finitely generated extension of a base field) is equivalent to the associated valuation being divisorial, and so, Abhyankar valuations rings that are not $F$-finite are abundant. The current paper rectifies the error in \cite{DS16}.  

On the other hand, the proof of \cite[Theorem 5.1]{DS16} does establish that a valuation $\nu$ of an $F$-finite function field $K/k$ is Abhyankar if and only if 
\begin{equation}
\label{1.0.1.1}
[\Gamma_\nu:p\Gamma_\nu][\kappa_\nu:\kappa_\nu^p] = [K:K^p].
\end{equation}
Said differently, $\nu$ is Abhyankar precisely when the extension of valuations $\nu/\nu^p$ is \emph{defectless}, where $\nu^p$ denotes the restriction of $\nu$ to the subfield $K^p$. We refer the reader to \cite[Page 281]{FV11} for the definition of defect of an extension of valuations. 

The second goal of this paper is to generalize the characterization of (\ref{1.0.1.1}) to valuations, not necessarily of function fields, that admit a Noetherian local center. When a valuation $\nu$ of an arbitrary field $K$ is centered on a Noetherian local domain $(R, \mathfrak m_R, \kappa_R)$ such that $\Frac(R) = K$, one has the following beautiful inequality established by Abhyankar \cite[Theorem 1]{Abh56}:
\begin{equation}
\label{1.0.1.2}
\dim_{\mathbb{Q}}(\mathbb{Q} \otimes_{\mathbb Z} \Gamma_\nu) + \td{\kappa_\nu/\kappa_R} \leq \dim{R}.
\end{equation}
When equality holds in (\ref{1.0.1.2}), $\nu$ behaves like an Abhyankar valuation of a function field. For example, the value group $\Gamma_\nu$ is then again a free abelian group of finite rank, and the residue field $\kappa_\nu$ is finitely generated over $\kappa_R$. However, whether a valuation of a function field is Abhyankar is intrinsic to the valuation, while equality in  (\ref{1.0.1.2}) with respect to a center depends, unsurprisingly, on the center as well (see Example \ref{non-excellent-DVR} for an illustration). Bearing this difference in mind, we call a Noetherian center $R$ an \emph{Abhyankar center} of $\nu$, if $\nu$ satisfies equality in (\ref{1.0.1.2}) with respect to $R$.

In practice one is often interested in centers satisfying additional restrictions. For example, in the local uniformization problem for valuations of function fields, one seeks centers that are regular and essentially of finite type over a ground field. Although satisfying equality in (\ref{1.0.1.2}) is not intrinsic to a valuation, the property of possessing Abhyankar centers from a more restrictive class of local rings may become independent of the center. For example, when $K/k$ is a function field and $\mathcal C$ is the collection of local rings that are essentially of finite type over $k$ with fraction field $K$, then a valuation $\nu$ admits an Abhyankar center from the collection $\mathcal C$ precisely when $\nu$ is an Abhyankar valuation of $K/k$. Moreover, then all centers of $\nu$ from $\mathcal C$ are Abhyankar centers (see Section \ref{Noetherian, F-finite centers}). In other words, the property of possessing Abhyankar centers that are essentially of finite type over $k$ is intrinsic to valuations of function fields over $k$. 

Our investigation reveals that even when the fraction field of a valuation is not a function field, one can find a reasonably broad class of Noetherian local rings such that the property of admitting an Abhyankar center from this class is  independent of the choice of the center. More precisely, we show the following:

\begin{customthm}{\ref{f-finiteness-general-abhyankar}}
Let $(R, \mathfrak{m}_R, \kappa_R)$ be a Noetherian $F$-finite local domain of characteristic $p > 0$ and fraction field $K$. Suppose $\nu$ is a non-trivial valuation of $K$ centered on $R$ with value group $\Gamma_\nu$ and valuation ring $(V, \mathfrak{m}_\nu, \kappa_\nu)$.  Then $R$ is an Abhyankar center of $\nu$ if and only if $$[\Gamma_\nu:p\Gamma_\nu][\kappa_\nu:\kappa^p_\nu] = [K:K^p].$$
\end{customthm}

\noindent Since the identity $[\Gamma_\nu:p\Gamma_\nu][\kappa_\nu:\kappa^p_\nu] = [K:K^p]$ does not depend on the center $R$, Theorem \ref{f-finiteness-general-abhyankar} implies that possessing $F$-finite Abhyankar centers is intrinsic to a valuation. 

Although imposing finiteness of Frobenius on Noetherian centers appears to be a strong restriction, the class of such rings include any center one is likely to encounter is geometric applications. For example, any generically $F$-finite excellent domain of prime characteristic is also $F$-finite. Thus Theorem \ref{f-finiteness-general-abhyankar} implies that the property of a valuation possessing generically $F$-finite, excellent Abhyankar centers is independent of the choice of such centers. 

We can draw interesting conclusions from Theorem \ref{f-finiteness-general-abhyankar} even for valuations of function fields. First note that it generalizes (\ref{1.0.1.1}) (see Remark \ref{Abhyankar-center-remarks}(i)). Furthermore, if $K$ is a function field over a perfect field $k$ of prime characteristic, then Theorem \ref{f-finiteness-general-abhyankar} and (\ref{1.0.1.1}) imply that any valuation of $K/k$ that possesses an excellent Abhyankar center is an Abhyankar valuation of $K/k$ (Corollary \ref{excellent-Abhyankar-centers-function-fields}). Admitting excellent Abhyankar centers is a priori much weaker than admitting Abhyankar centers that are essentially of finite type over $k$, and so Corollary \ref{excellent-Abhyankar-centers-function-fields} is not obvious. In fact, the corresponding statement is false when the ground field $k$ has characteristic $0$ -- one can easily construct a non-Abhyankar valuation of $K/\mathbb{C}$ which has an Abhyankar center that is an excellent local domain (see Remark \ref{Abhyankar-center-remarks}(vi) or \cite[Example 1(iv)]{ELS03}).

Theorem \ref{f-finiteness-general-abhyankar} \emph{does not} claim that if a valuation $\nu$ of an $F$-finite field $K$ satisfies $[\Gamma_\nu:p\Gamma_\nu][\kappa_\nu:\kappa^p_\nu] = [K:K^p]$, then $\nu$ is centered on an excellent local domain; the latter assertion is false even when $K$ is not perfect. Indeed, using the theory of $F$-singularities of valuations developed in \cite{DS16, DS17}, one can show that if the valuation ring of $\nu$ is $F$-finite, then the identity $[\Gamma_\nu:p\Gamma_\nu][\kappa_\nu:\kappa^p_\nu] = [K:K^p]$ is always satisfied (Section \ref{Summary of F-singularities for valuation rings}, (\ref{assuming-f-finiteness})(ii)). However, we prove in this paper that a non-Noetherian $F$-finite valuation ring \emph{cannot} be centered on any excellent local domain of its fraction field (Section \ref{Summary of F-singularities for valuation rings}, (\ref{non-Noetherian-$F$-finite-rings})), and explicitly construct such a valuation ring of $K = L(X)$, where $L$ is a perfect field admitting a non-trivial valuation (Example \ref{valuation-non-perfect-no-F-finite-center}). The valuation ring of Example \ref{valuation-non-perfect-no-F-finite-center} \emph{does not} contain the ground field $L$, and so its existence does not contradict  local uniformization in positive characteristic. On the contrary, this paper and the author's prior work with Karen Smith \cite[Theorem 0.1]{DS17} shows it is impossible to construct non-Noetherian $F$-finite valuation rings of function fields when the valuation rings contain the ground field (see also Remark \ref{function-vs-arbitrary-field}). Thus, pathologies such as Example \ref{valuation-non-perfect-no-F-finite-center} cease to exist for valuations trivial on the ground field of a function field.


In the final section of this paper (Section \ref{Summary of F-singularities for valuation rings}), we summarize all known results on F-singularities of valuation rings for the convenience of the reader, mainly drawing from the paper \cite{DS16} and the accompanying corrigendum \cite{DS17}. Results are grouped according to the type of $F$-singularity they characterize, which we hope will make it easier for the reader to navigate \cite{DS16, DS17}.  The summary is not just limited to a recollection of old results; many new results are also proved. For example, we show that when $\nu/\nu^p$ is \emph{totally unramified} or has maximal defect, that is, when 
$$[\Gamma_\nu:p\Gamma_\nu][\kappa_\nu:\kappa_\nu^p] = 1,$$ 
then the valuation ring of $\nu$ cannot be Frobenius split (Section 5, (\ref{totally-unramified})). To put this in context, Theorem \ref{Theorem-A} implies, in contrast, that if $\nu/\nu^p$ is \emph{defectless} for a valuation $\nu$ of a function field over a perfect field, then its valuation ring is always Frobenius split. Thus, Frobenius splitting of valuation rings is well-understood when the defect of $\nu/\nu^p$ is one of two possible extremes, but seems to be mysterious otherwise. We also show that if $V$ is valuation ring of Krull dimension $1$ whose fraction field is spherically complete with respect to the multiplicative norm induced by the valuation, then $V$ is Frobenius split (Section 5, (\ref{spherical-complete})). This result should be viewed as a generalization of the well-known fact that a complete discrete valuation ring of prime characteristic is Frobenius split.

\textbf{Acknowledgments:} The question of Frobenius splitting of valuation rings arose in conversations of Zsolt Patakfalvi, Karl Schwede and Karen Smith. While Frobenius splitting of Abhyankar valuations was suspected in my prior work with Karen, the idea of using local uniformization took shape while I was visiting University of Utah. I thank Karl Schwede, Raymond Heitmann, Linquan Ma and Anurag Singh for many fruitful conversations during my stay in Utah, and Karl for the invitation. I am also grateful to Karen, Linquan and Raymond Heitmann for helpful comments on a draft of this paper, and to Eric Canton, Takumi Murayama and Matthew Stevenson for allowing me to include (Section 5, (\ref{spherical-complete})) that we realized while working on a different project. Further thanks go to Franz-Viktor Kuhlmann, whose question on the relationship between defect and Abhyankar valuations inspired the material of Section \ref{Noetherian, F-finite centers}, and to Steven Dale Cutkosky for helpful conversations. My work was supported by department and summer fellowships at University of Michigan and NSF Grant DMS \#1501625.




\vspace{-1ex}








\section{Background and conventions}

\subsection{Valuations} \label{valuations-background} All valuations are written additively, and local rings are not necessarily Noetherian. We say a valuation $\nu$ with valuation ring $V$ is \emph{centered} on a local ring $A$, or $A$ is a \emph{center} of $\nu$ (or $V$) if $A \subseteq V$, and the maximal ideal of $V$ contracts to the maximal ideal of $A$. It is always assumed that the fraction field of a center coincides with the fraction field of the valuation ring. 

Let $K$ be a finitely generated field extension of a field $k$, that is $K$ is a \emph{function field} over $k$.  A valuation $\nu$ of $K/k$ (this means $\nu$ is trivial on $k$) with valuation ring $(R_{\nu}, \mathfrak{m}_{\nu}, \kappa_{\nu})$ and value group $\Gamma_{\nu}$ satisfies the fundamental inequality
\begin{equation}
\label{Abh-ineq}
\dim_{\mathbb{Q}}(\mathbb{Q} \otimes_{\mathbb Z} \Gamma_{\nu}) + \td{\kappa_{\nu}/k} \leq \td{K/k}.
\end{equation}
If equality holds in the above inequality, then $\nu$ is called an \textbf{Abhyankar valuation} or a \textbf{quasi-monomial valuation} of $K/k$, and the associated valuation ring of $\nu$ is called an \textbf{Abhyankar valuation ring} of $K/k$.

For Abhyankar valuations it is well-known \cite[VI, $\mathsection 10.3$, Corollary 1]{Bou89} that $\Gamma_{\nu}$ is a \emph{free abelian group} of finite rank equal to 
$$d \coloneqq \dim_{\mathbb{Q}}(\Gamma_{\nu} \otimes_{\mathbb{Z}} \mathbb{Q}),$$
and the residue field $\kappa_{\nu}$ is a \emph{finitely generated field extension} of $k$ of transcendence degree $n - d$, where 
$$n \coloneqq \td{K/k}.$$
If the ground field $k$ has prime characteristic $p$ and $[k:k^p] < \infty$, this implies 
$$\textrm{$[K:K^p] = [k:k^p]p^n$ and $[\kappa_{\nu}:\kappa^p_{\nu}] = [k:k^p]p^{n-d}$}.$$ An important fact, used implicitly throughout the paper, is that if $x, y \in K^{\times}$ are non-zero elements such that $x+y \neq 0$ and $\nu(x) \neq \nu(y)$, then 
$$\nu(x+ y) = \inf\{\nu(x), \nu(y)\}.$$ 
A \emph{discrete valuation ring} (abbreviated DVR) is a Noetherian valuation ring which is not a field.  Equivalently, a DVR is a dimension $1$ regular local ring.

\subsection{Frobenius} Let $R$ be a ring of prime characteristic. We have the (absolute) Frobenius map
$$F: R \rightarrow R,$$
which maps an element $r \in R$ to its $p$-th power. The target copy of $R$ is usually considered as an $R$-module by restriction of scalars via $F$, and is then denoted $F_*R$. In other words, if $r \in R$ and $x \in F_*R$, then $r \cdot x = r^px$. For an ideal $I$ of $R$, by $I^{[p^e]}$ we mean the ideal generated by the $p^e$-th powers of elements of $I$. Thus $I^{[p^e]}$ is the expansion of $I$ under $F^e$, the $e$-th iterate of Frobenius.

Quite remarkably, the Frobenius map can detect when a Noetherian ring is regular, and the foundational result in the theory of $F$-singularities is the following:

\begin{theorem}[\cite{Kun69}, Theorem 2.1]
\label{Kunz-theorem}
Let $R$ be a Noetherian ring of prime characteristic. Then $R$ is regular if and only if $F: R \rightarrow F_*R$ is a flat ring homomorphism.
\end{theorem}
\noindent If Frobenius is a finite map and $R$ is regular local, the above theorem implies that $F_*R$ is a  free $R$-module. The freeness of $F_*R$ will be important when proving Theorem \ref{Theorem-A}.


In geometry, finiteness of Frobenius is a mild restriction. For instance, Frobenius is a finite map for the localization of any finitely generated algebra over a perfect field of prime characteristic, and also for any complete local ring whose residue field $k$ satisfies $[k:k^p] < \infty$. A ring for which Frobenius is finite is called \textbf{$F$-finite}.

Kunz's theorem shows that when $R$ is regular local and $F$-finite, $F_*R$ has many free $R$-summands. For an arbitrary ring $R$, if $F_*R$ has at least one free $R$-summand, we say $R$ is \emph{Frobenius split}. More formally, $R$ is \textbf{Frobenius split} if $F: R \rightarrow F_*R$ has a left inverse, called a \textbf{Frobenius splitting}, in the category of $R$-modules. 
When $R$ is reduced, Frobenius is an isomorphism onto its image $R^p$, and the existence of a Frobenius splitting is equivalent to the existence of an $R^p$-linear map $R \rightarrow R^p$ that maps $1 \mapsto 1$. 

Weakening the notion of Frobenius splitting leads to \emph{$F$-purity}. We say that $R$ is \textbf{$F$-pure} when Frobenius is a pure map of $R$-modules. This means for any $R$-module $M$, the induced map $ F\otimes \textrm{id}_M: M \rightarrow F_*R \otimes_R M$ is injective. Regular rings are $F$-pure because Frobenius is flat hence faithfully flat for such rings, and faithful flatness implies purity. Also Frobenius splitting clearly implies $F$-purity, but the converse is false. For example, any non-excellent DVR of a function field over a perfect ground field is F-pure but not Frobenius split. See (\ref{flatness}) and (\ref{noetherian-Frobenius-splitting}) in Section \ref{Summary of F-singularities for valuation rings} for further discussion, and Example \ref{non-excellent-DVR} for a non-excellent DVR of a function field.



\section{Proof of Theorem \ref{Theorem-A}}
\label{Proof of Theorem Theorem-A}
Unless otherwise specified, throughout this section we assume that $K$ is a finitely generated field extension of an $F$-finite ground field $k$, and $\nu$ is an Abhyankar valuation of $K/k$ whose residue field $\kappa_{\nu}$ is separable over $k$. \emph{We will prove more generally under these assumptions that the valuation ring $R_{\nu}$ is Frobenius split.} 

Also, we let 
$$\text{$d \coloneqq \dim_{\mathbb Q}(\mathbb{Q}\otimes_{\mathbb{Z}} \Gamma_{\nu})$ and $n \coloneqq \td{K/k}$.}$$
\noindent The goal is to choose a regular local center of ${\nu}$ satisfying some nice properties, and then extend a Frobenius splitting of this center to a Frobenius splitting of $R_{\nu}$. The center we seek is given by the following lemma.

\begin{lemma}
\label{regular-center}
Let $\nu$ be an Abhyankar valuation as in Theorem \ref{local-uniformization}. Then there exists a regular local ring $(A, \mathfrak{m}_A, \kappa_A)$ which is essentially of finite type over $k$ with fraction field $K$ satisfying the following properties:
\begin{enumerate}
\item $R_\nu$ is centered on $A$, and $\kappa_A \hookrightarrow \kappa_\nu$ is an isomorphism.
\item $A$ has Krull dimension $d$, and there exist a regular system of parameters $\{x_1, \dots, x_d\}$ of $A$ such that $\{\nu(x_1), \dots, \nu(x_d)\}$ freely generates the value group $\Gamma_\nu$.
\end{enumerate}
\end{lemma}

\begin{proof}
Choose $r_1, \dots, r_d, s_1, \dots, s_t \in R_\nu$ such that $\{\nu(r_1), \dots, \nu(r_d)\}$ freely generates the value group $\Gamma_\nu$, and the images of $s_1, \dots, s_t$ in $\kappa_\nu$ generate the latter over $k$. Taking 
$$Z \coloneqq \{r_1, \dots, r_d, s_1, \dots, s_t\},$$
by local monomialization (Theorem \ref{local-uniformization}) there exists a regular local ring $(A, \mathfrak{m}_A, \kappa_A)$  essentially of finite type over $k$ with fraction field $K$ such that $R_\nu$ dominates $A$, $A$ has dimension $d$, $Z \subseteq A$, and there exists a regular system of parameters $\{x_1, \dots, x_d\}$ of $A$ such that every $z \in Z$ admits a factorization
$$z = ux^{a_1}_1\dots x^{a_d}_d,$$
for some $u \in A^{\times}$, and $a_i \in \mathbb{N} \cup \{0\}$.
In particular, each $\nu(r_i)$ is $\mathbb Z$-linear combination of $\nu(x_1), \dots, \nu(x_d)$, which implies that $\{\nu(x_1), \dots, \nu(x_d)\}$ also freely generates $\Gamma_\nu$. Moreover, our choice of $s_1, \dots, s_t$ implies that $\kappa_A = \kappa_\nu$.
\end{proof}

\begin{remark}
For a valuation $\nu$ of $K/k$, the existence of a center which is an essentially of finite type $k$-algebra of Krull dimension equal to $\dim_{\mathbb Q} (\mathbb{Q} \otimes_{\mathbb Z} \Gamma_\nu)$ implies that $\nu$ is Abhyankar. Thus, only Abhyankar valuations admit a center as in Lemma \ref{regular-center}.
\end{remark}

From now on $A$ will denote a choice of a regular local center of $\nu$ that satisfies Lemma \ref{regular-center}, and $\{x_1, \dots, x_d\}$ a regular system of parameters of $A$ whose valuations freely generate $\Gamma_\nu$. Observe that $A$ is $F$-finite since it is essentially of finite type over an $F$-finite field. Then Theorem \ref{Kunz-theorem} implies that $A$ is free over its $p$-th power subring $A^p$ of rank equal to $[K:K^p] = [k:k^p]p^n$. For 
$$f \coloneqq [\kappa_\nu:\kappa^p_\nu] = [k:k^p]p^{n-d},$$ 
if we choose 
$$1 = y_1,\hspace{1mm} y_2,\hspace{1mm} \dots,\hspace{1mm} y_f \in A,$$ 
such that the images of $y_i$ in $\kappa_A = \kappa_\nu$ form a basis of $\kappa_\nu$ over $\kappa_\nu^p$, then
$$\mathcal{B} \coloneqq \{y_jx_1^{\beta_1}\dots x_d^{\beta_d}: 1\leq j \leq f,\hspace{1mm} 0 \leq \beta_i \leq p-1\},$$
is a free basis of $A$ over $A^p$. Note the elements $y_j$ are units in $A$.

With respect to the basis $\mathcal B$, $A$ has a natural Frobenius splitting
$$\eta_{\mathcal B}: A \rightarrow A^p,$$
given by mapping $1 = y_1x^0_1\dots x^0_d \mapsto 1$, and all the other basis elements to $0$. Extending $\eta_{\mathcal B}$ uniquely to a $K^{p}$-linear map 
$$\widetilde{\eta_{\mathcal B}}: K \rightarrow K^p$$ 
of the fraction fields, we will show that the restriction of $\widetilde{\eta_{\mathcal B}}$ to $R_\nu$ yields a Frobenius splitting of $R_\nu$, or in other words, $\widetilde{\eta_{\mathcal B}}|_{R_\nu}$ maps into $R^p_\nu$.

\begin{claim}
\label{main-claim}
For any $a \in A$, either $\eta_{\mathcal B}(a) = 0$ or $\nu(\eta_{\mathcal{B}}(a)) \geq \nu(a)$.
\end{claim}

Theorem \ref{Theorem-A} follows easily from the claim using the following general observation. 

\begin{lemma}
\label{criterion-for-F-splitting-valuation-rings}
Let $\nu$ be a valuation of a field $K$ of characteristic $p > 0$ with valuation ring $R_\nu$, and $A$ a subring of $R_\nu$ such that $\Frac(A) = K$. Suppose $\varphi: A \rightarrow A^{p^e}$ is an $A^{p^e}$-linear map, for some $e \geq 1$. Consider the following:
\begin{enumerate}
\item[(i)] For all $a \in A$, \textrm{$\varphi(a) = 0$ or $\nu(\varphi(a)) \geq \nu(a)$}.
\item[(ii)] For all $a, b \in A$ such that $\nu(a) \geq \nu(b)$, if $\varphi(ab^{p^e -1}) \neq 0$, then $\nu(\varphi(ab^{p^e-1})) \geq \nu(b^{p^e})$.
\item[(iii)] $\varphi$ extends to an $R^{p^e}_\nu$-linear map $R_\nu \rightarrow R^{p^e}_\nu$.
\item[(iv)] $\varphi$ extends uniquely to an $R^{p^e}_\nu$-linear map $R_\nu \rightarrow R^{p^e}_\nu$.
\end{enumerate}
Then (ii), (iii) and (iv) are equivalent, and (i) $\Rightarrow$ (ii). Moreover, if $\varphi$ is a Frobenius splitting of $A$ satisfying (i) or (ii), then $\varphi$ extends to a Frobenius splitting of $R_\nu$.
\end{lemma}

\begin{proof}
For the final assertion on Frobenius splitting, note that the extension of a Frobenius splitting remains a Frobenius splitting since $1 \mapsto 1$ in the extension.

(i) $\Rightarrow$ (ii): If $\varphi(ab^{p^e -1}) \neq 0$, we have 
$$\nu(\varphi(ab^{p^e-1})) \geq \nu(ab^{p^e-1}) \geq \nu(b^{p^e}),$$
where the first inequality follows from (i), and the second inequality follows from $\nu(a) \geq \nu(b)$.

(ii) $\Rightarrow$ (iii): Extending $\varphi$ to a $K^{p^e}$-linear map $\widetilde{\varphi}: K \rightarrow K^p$, it suffices to show that $\widetilde{\varphi}|_{R_\nu}$ maps into $R^{p^e}_\nu$. Let $r \in R_\nu$ be a non-zero element. Since $K$ is the fraction field of $A$ and $R_\nu$, one can express $r$ as a fraction $a/b$, for non-zero $a, b \in A$. Note  
$$ 
\nu(a) \geq \nu(b).
$$
Then 
\begin{equation}
\label{main-equation}
\widetilde{\varphi}(r) = \widetilde{\varphi}\bigg{(}\frac{a}{b}\bigg{)} = \frac{1}{b^{p^e}}\varphi(ab^{p^e-1}).
\end{equation}
If $\varphi(ab^{p^e-1}) = 0$, then $\widetilde{\varphi}(r) = 0$, and $r$ maps into $R^{p^e}_\nu$. Otherwise by assumption,
$$\nu(\varphi(ab^{p^e-1})) \geq \nu(b^{p^e}),$$
and so, 
$$\nu(\widetilde{\varphi}(r)) = \nu(\varphi(ab^{p^e-1})) - \nu(b^{p^e}) \geq 0,$$
that is $\widetilde{\varphi}(r)$ is an element of $K^{p^e} \cap R_\nu = R^{p^e}_\nu$. 

(iii) $\Rightarrow$ (iv): Since $A$ and $R_\nu$ have the same fraction field, any extension of $\varphi$ to $R_\nu$ is obtained as a restriction to $R_\nu$ of the unique extension of $\varphi$ to a $K^{p^e}$-linear map $\widetilde{\varphi}: K \rightarrow K^{p^e}$, and so is also unique. See (\ref{main-equation}) above for a concrete description of how $\varphi$ extends to $R_\nu$. 

To finish the proof of the lemma, it suffices to show (iv) $\Rightarrow$ (ii). But this also follows easily from (\ref{main-equation}).
\end{proof}

\begin{proof}[\textbf{Proof of Claim \ref{main-claim}}]
Recall that 
$$\mathcal{B} = \{y_jx_1^{\beta_1}\dots x_d^{\beta_d}: 1\leq j\leq f, \hspace{1mm} 0 \leq \beta_i \leq p-1\}$$ 
is a basis of $A$ over $A^p$, where the $x_i$ and $y_j$ are chosen such that $\{\nu(x_1), \dots, \nu(x_d)\}$ freely generates the value group $\Gamma_\nu$, and the images of $1 = y_1, \hspace{1mm} y_2, \hspace{1mm} \dots,\hspace{1mm} y_f$ in $\kappa_\nu$ form a basis of $\kappa_\nu$ over $\kappa^p_\nu$. The $A^p$-linear Frobenius splitting $\eta_{\mathcal B}$ is given by
$$\eta_{\mathcal B}\bigg{(}\sum_{j = 1}^f \sum_{0 \leq \beta_i \leq p-1}c^p_{j, \beta_1, \dots, \beta_d}y_jx_1^{\beta_1}\dots x_d^{\beta_d} \bigg{)} = c^p_{1,0, 0, \dots, 0}.$$
Thus, we need to show that either $c^p_{1,0, 0, \dots, 0} = 0$ or
$$\nu(c^p_{1,0, 0, \dots, 0}) \geq \nu\bigg{(}\sum_{j = 1}^f \sum_{0 \leq \beta_i \leq p-1}c^p_{j, \beta_1, \dots, \beta_d}y_jx_1^{\beta_1}\dots x_d^{\beta_d} \bigg{)}.$$
Assuming without loss of generality that $\sum_{j = 1}^f \sum_{0 \leq \beta_i \leq p-1}c^p_{j, \beta_1, \dots, \beta_d}y_jx_1^{\beta_1}\dots x_d^{\beta_d} \neq 0$, we will prove the stronger fact that
\begin{equation}
\label{inf-eq}
\nu\bigg{(}\sum_{j = 1}^f \sum_{0 \leq \beta_i \leq p-1}c^p_{j, \beta_1, \dots, \beta_d}y_jx_1^{\beta_1}\dots x_d^{\beta_d} \bigg{)} = \inf\{\nu(c^p_{j,\beta_1,\dots,\beta_d}y_jx^{\beta_1}_1\dots x^{\beta_d}_d): c^p_{j,\beta_1,\dots,\beta_d} \neq 0\}.
\end{equation}

For two non-zero terms $c^p_{j,\alpha_1,\dots,\alpha_d}y_jx_1^{\alpha_1}\dots x_d^{\alpha_d}$ and $c^p_{k,\beta_1,\dots,\beta_d}y_kx^{\beta_1}_1\dots x^{\beta_d}_d$ in the above sum,
\begin{equation}
\label{equation}
\nu(c^p_{j,\alpha_1,\dots,\alpha_d}y_jx_1^{\alpha_1}\dots x_d^{\alpha_d}) = \nu(c^p_{k,\beta_1,\dots,\beta_d}y_kx^{\beta_1}_1\dots x^{\beta_d}_d)
\end{equation}
if and only if 
\begin{equation}
\label{another-equality}
p\nu(c_{j,\alpha_1,\dots,\alpha_d}) + \alpha_1\nu(x_1) + \dots + \alpha_d\nu(x_d) = p\nu(c_{k, \beta_1,\dots,\beta_d}) + \beta_1\nu(x_1) + \dots + \beta_d\nu(x_d).
\end{equation}
By $\mathbb Z$-linear independence of $\nu(x_1),\hspace{1mm} \dots,\hspace{1mm} \nu(x_d)$, for all $i = 1, \dots, d$, we get 
$$p|(\alpha_i - \beta_i).$$ 
Since $0 \leq \alpha_i, \beta_i \leq p-1$, this means that $\alpha_i = \beta_i$ for all $i$, and moreover, then $\nu(c^p_{j,\alpha_1,\dots,\alpha_d}) = \nu(c^p_{k,\beta_1,\dots,\beta_d})$. Thus, (\ref{equation}) holds precisely when $\nu(c^p_{j,\alpha_1,\dots,\alpha_d}) = \nu(c^p_{k, \beta_1,\dots,\beta_d})$ and $\alpha_i = \beta_i$, for all $i = 1, \dots, d$.

For ease of notation, let us use $\underline{\alpha}$ as a shorthand for $\alpha_1, \dots, \alpha_d$, and $\underline{x}^{\underline{\alpha}}$ for $x^{\alpha_1}_1\dots x^{\alpha_d}_d$. Then for a fixed non-zero term $c^p_{j_1,\underline{\alpha}}y_{j_1}\underline{x}^{\underline{\alpha}}$, consider the set 
$$\{c^p_{j_1,\underline{\alpha}}y_{j_1}\underline{x}^{\underline{\alpha}}, \hspace{1mm} c^p_{j_2,\underline{\alpha}}y_{j_2} \underline{x}^{\underline{\alpha}}, \dots, \hspace{1mm} c^p_{j_i, \underline{\alpha}}y_{j_i}\underline{x}^{\underline{\alpha}}\}$$ 
of all non-zero terms of $\sum_{j = 1}^f \sum_{0 \leq \beta_i \leq p-1}c^p_{j, \beta_1, \dots, \beta_d}y_jx_1^{\beta_1}\dots x_d^{\beta_d}$
having the same valuation as $c^p_{j_1,\underline{\alpha}}y_{j_1}\underline{x}^{\underline{\alpha}}$. In particular, by the above reasoning we also have
$$\nu(c^p_{j_1, \underline{\alpha}}) = \nu(c^p_{j_2,\underline{\alpha}}) = \dots = \nu(c^p_{j_i,\underline{\alpha}}).$$ Adding these terms of equal valuation, in the valuation ring $R_\nu$ one can write
\begin{align*}
c^p_{j_1,\underline{\alpha}}y_{j_1}\underline{x}^{\underline{\alpha}} + c^p_{j_2,\underline{\alpha}}y_{j_2} \underline{x}^{\underline{\alpha}}+ \dots + c^p_{j_i,\underline{\alpha}}y_{j_i}\underline{x}^{\underline{\alpha}} =\\ 
\bigg{(}y_{j_1} + \bigg{(}\frac{c_{j_2,\underline{\alpha}}}{c_{j_1,\underline{\alpha}}}\bigg{)}^py_{j_2} + \dots + \bigg{(}\frac{c_{j_i,\underline{\alpha}}}{c_{j_1,\underline{\alpha}}}\bigg{)}^py_{j_i}\bigg{)}c_{j_1,\underline{\alpha}}^p\underline{x}^{\underline{\alpha}},
\end{align*}
where 
$$y_{j_1} + \bigg{(}\frac{c_{j_2,\underline{\alpha}}}{c_{j_1,\underline{\alpha}}}\bigg{)}^py_{j_2} + \dots + \bigg{(}\frac{c_{j_i,\underline{\alpha}}}{c_{j_1,\underline{\alpha}}}\bigg{)}^py_{j_i}$$
is a unit in $R_\nu$ by the $\kappa^p_\nu$-linear independence of the images of $y_{j_1}, \hspace{1mm} \dots, \hspace{1mm} y_{j_i}$ in $\kappa_\nu$ and the fact that $({c_{j_2,\underline{\alpha}}}/{c_{j_1,\underline{\alpha}}})^p, \dots, ({c_{j_i,\underline{\alpha}}}/{c_{j_1,\underline{\alpha}}})^p$ are units in $R_\nu^p$. Thus, the valuation of the sum 
$$c^p_{j_1,\underline{\alpha}}y_{j_1}\underline{x}^{\underline{\alpha}} + \dots + c^p_{j_i,\underline{\alpha}}y_{j_i}\underline{x}^{\underline{\alpha}}$$ 
equals the valuation of any of its terms. Now rewriting $\sum_{j = 1}^f \sum_{0 \leq \beta_i \leq p-1}c^p_{j, \beta_1, \dots, \beta_d}y_jx_1^{\beta_1}\dots x_d^{\beta_d}$ by collecting non-zero terms having the same valuation, (\ref{inf-eq}), hence also the claim, follows.
\end{proof}



\vspace{2mm}

\begin{examples}
\label{some-examples}
{\*}
\item (a) A valuation ring of a function field of a curve over an $F$-finite ground field is always Frobenius split. Indeed, such a valuation ring is always centered on some normal affine model of dimension $1$ of the function field, and so is an $F$-finite DVR.

\item (b) For a positive integer $n$, consider $\mathbb Z^{\oplus n}$ with the lexicographical order. That is, if $\{e_1, \dots, e_n\}$ denotes the standard basis of $\mathbb Z^{\oplus n}$, then $e_1 > e_2 > \dots > e_n$. There exists a unique valuation $\nu_{lex}$ on $\mathbb{F}_p(x_1, \dots, x_n)/\mathbb{F}_p$ such that for all $i \in \{1, \dots, n\}$,
$$\nu_{lex}(x_i) = e_i.$$
The valuation $\nu_{lex}$ is clearly Abhyakar since $\dim_{\mathbb{Q}}(\mathbb{Q} \otimes_{\mathbb{Z}} \mathbb{Z}^{\oplus n}) = n$, which coincides with the transcendence degree of $\mathbb{F}_p(x_1, \dots, x_n)/\mathbb F_p$. One can also show that the valuation ring $R_{\nu_{lex}}$ has Krull dimension $n$ and residue field $\mathbb{F}_p$. The valuation is centered on the regular local ring $\mathbb{F}_p[x_1,\dots,x_n]_{(x_1,\dots, x_n)}$ such that the valuations of the obvious regular system of parameters freely generate $\mathbb{Z}^{\oplus n}$ and the residue field coincides with the residue field of $\nu_{lex}$. Then a Frobenius splitting of $R_{\nu_{lex}} \rightarrow R^p_{\nu_{lex}}$ is obtained by extending the canonical splitting on $\mathbb{F}_p[x_1,\dots,x_n]_{(x_1,\dots,x_n)}$ with respect to the basis 
$$\{x_1^{\beta_1}\dots x_n^{\beta}: 0 \leq \beta_i \leq p-1\}.$$
This splitting of $\mathbb{F}_p[x_1,\dots,x_n]_{(x_1,\dots,x_n)}$ maps
\[   
 x_1^{\alpha_1}\dots x_n^{\alpha_n} \mapsto 
     \begin{cases}
       x_1^{\alpha_1}\dots x_n^{\alpha_n} &\quad\text{if $p|\alpha_i$ for all $i$,}\\
       0 &\quad\text{otherwise.} \\ 
     \end{cases}
\]

\item (c) Let $\Gamma = \mathbb{Z} \oplus \mathbb{Z}\pi \subset \mathbb{R}$. There exists a valuation $\nu$ of $\mathbb{F}_p(x,y,z)/\mathbb{F}_p$ given by
$$\nu(x) = \nu(y) = 1, \nu(z) = \pi.$$
Then $\dim_{\mathbb Q} (\mathbb{Q} \otimes_{\mathbb Z} \Gamma) = 2$, and the transcendence degree of the residue field $\kappa_\nu/\mathbb{F}_p$ is at least $1$ since the image of $y/x$ in the residue field is transcendental over $\mathbb{F}_p$. Therefore the fundamental inequality (\ref{Abh-ineq}) implies that $\nu$ is Abhyankar. Although the valuation $\nu$ is centered on the regular local ring $\mathbb{F}_p[x,y,z]_{(x,y,z)}$, no regular system of parameters can freely generate the value group because the center has dimension $3$, whereas the value group is free of rank $2$. However, blowing up the origin in $\mathbb{A}^3_{\mathbb F_p}$, we see that $\nu$ is now centered on the regular local ring $\mathbb{F}_p[x, y/x, z/x]_{(x, z/x)}$, and the valuations of the regular system of paramaters $\{x, z/x\}$ freely generate $\Gamma$. Furthermore, the residue field of $\mathbb{F}_p[x, y/x, z/x]_{(x, z/x)}$ can be checked to coincide with the residue field of the valuation ring. Relabelling $y/x$ and $z/x$ as $u, w$ respectively, a Frobenius splitting on $R_\nu$ is obtained by extending the Frobenius splitting of $\mathbb{F}_p[x, u, w]_{(x, w)}$ given by the same rule as in (a) with respect to the transcendental elements $x, u, w$ over $\mathbb F_p$.
\end{examples}

\section{Valuations centered on Noetherian, local domains}
\label{Noetherian, F-finite centers}

The proof of \cite[Theorem 5.1]{DS16} shows that a valuation $\nu$ of an $F$-finite function field $K/k$ of characteristic $p$ is Abhyankar precisely when
\begin{equation}
\label{function-field-defectless}
[\Gamma_\nu:p\Gamma_\nu][\kappa_\nu:\kappa^p_\nu] = [K:K^p].
\end{equation}
If $\nu^p$ denotes the restriction of $\nu$ to the subfield $K^p$ of $K$, then the value group of $\nu^p$ is easily verified to be $p\Gamma_\nu$, and the residue field $\kappa_{\nu^p}$ can be identified with $\kappa_\nu^p$. Thus, $[\Gamma_\nu:p\Gamma_\nu]$ is the \emph{ramification index}, and $[\kappa_\nu:\kappa_\nu^p]$ the \emph{residue degree} of the extension of valuations $\nu/\nu^p$ \cite[Remark 4.3.3]{DS16}. Note $\nu$ is the \emph{unqiue} extension of $\nu^p$ to $K$ since $K$ is a purely inseparable extension of $K^p$.

In terms of the theory of extensions of valuations, (\ref{function-field-defectless}) can be reinterpreted as saying that a necessary and sufficient condition for $\nu$ to be Abhyankar is for the unique extension of valuations $\nu/\nu^p$ to be \emph{defectless} (see \cite[Page 281]{FV11} for the definition of defect). There is a natural generalization of the notion of an Abhyankar valuation for valuations of arbitrary fields. The goal of this section is to introduce this more general notion, and investigate to what extent such valuations can be characterized in terms of the defect of $\nu/\nu^p$.

We fix some notation. Let $K$ denote a field of characteristic $p > 0$ (not necessarily a function field), and $\nu$ a valuation of $K$ with valuation ring $(V, \mathfrak{m}_{\nu}, \kappa_\nu)$ centered on a Noetherian local ring $(R, \mathfrak{m}_R, \kappa_R)$ such that
$$\dim(R) < \infty.$$ 
Recall that centers, by convention, always have the same fraction field as the valuation ring. Let $\Gamma_\nu$ be the value group of $\nu$.

Abhyankar greatly generalized (\ref{Abh-ineq}) in \cite[Theorem 1]{Abh56}, establishing that in the above setup, 
\begin{equation}
\label{more-general-Abhyankar-ineq}
\dim_{\mathbb Q}(\mathbb Q \otimes_{\mathbb Z} \Gamma_\nu) + \td{\kappa_\nu/\kappa_R} \leq \dim(R).
\end{equation}
Moreover, he showed that if equality holds in the above inequality, then $\Gamma_\nu$ is a free abelian group and $\kappa_\nu$ is a finitely generated extension of $\kappa_R$. 

When $K/k$ is a function field, there is a close relationship between Abhyankar valuations of $K/k$, and those valuations of $K/k$ that admit a Noetherian center with respect to which equality holds in (\ref{more-general-Abhyankar-ineq}). Indeed, if equality holds in (\ref{more-general-Abhyankar-ineq}) for an arbitrary valuation $\nu$ of $K/k$ with respect to a center $R$ which is essentially of finite type over $k$, then $\nu$ is an Abhyankar valuation of $K/k$. To see this, let $n = \td{K/k}$. Then 
$$\td{\kappa_R/k} = n-\dim(R),$$ 
because $R$ is essentially of finite type over $k$ with fraction field $K$, and so 
$$\td{\kappa_\nu/k} = \td{\kappa_\nu/\kappa_R} + n - \dim(R).$$
This implies  
\begin{align*}
\dim_{\mathbb{Q}}(\mathbb{Q} \otimes_{\mathbb{Z}} \Gamma_\nu)  + \td{\kappa_\nu/k} & = (\dim_{\mathbb{Q}}(\mathbb{Q} \otimes_{\mathbb{Z}} \Gamma_\nu)  +  \td{\kappa_\nu/\kappa_R})  +  n  -  \dim(R) \\
&= \dim(R)  +  n  -  \dim(R) = \td{K/k}.
\end{align*}
Conversely, a similar reasoning shows that if $\nu$ is an Abhyankar valuation of $K/k$, then $\nu$ satisfies equality in (\ref{more-general-Abhyankar-ineq}) with respect to any center which is essentially of finite type over $k$.

However, despite the similarity between (\ref{more-general-Abhyankar-ineq}) and (\ref{Abh-ineq}), whether a valuation satisfies equality in (\ref{more-general-Abhyankar-ineq}) is not an intrinsic property of the valuation, but also depends on the center $R$. In contrast, the property of being an Abhyankar valuation is intrinsic to valuations of function fields. To better illustrate this difference, we construct a valuation of $\mathbb{F}_p(X,Y)$ with two different Noetherian centers such that equality in (\ref{more-general-Abhyankar-ineq}) is satisfied with respect to one center, but not the other. In our example we work over a base field of characteristic $p > 0$, but the construction goes through when the ground field has characteristic $0$.

\begin{example}(see also \cite[Example 4.0.5]{DS16})
\label{non-excellent-DVR}
Consider the laurent series field $\mathbb{F}_p((t))$ with the canonical $t$-adic valuation, $\nu_{t}$, whose corresponding valuation ring is the DVR $\mathbb{F}_p[[t]]$. Choose an embedding of fields
$$i: \mathbb{F}_p(X,Y) \hookrightarrow \mathbb{F}_p((t))$$
by mapping $X \mapsto t$ and $Y \mapsto p(t)$, where 
$$p(t) \in \mathbb{F}_p[[t]]$$ 
such that $\{t, p(t)\}$ are algebraically independent over $\mathbb{F}_p$. Such a power-series exists because  $\mathbb{F}_p((t))$ is uncountable, but $\mathbb{F}_p(t)$ is countable. Moreover, multiplying $p(t)$ by $t$, we may even assume that $t|p(t)$. Then we get a new valuation $\nu$ on $\mathbb{F}_p(X,Y)$, given by the composition
$$\nu \coloneqq \mathbb{F}_p(X,Y)^{\times} \xrightarrow{i} \mathbb{F}_p((t))^{\times} \xrightarrow{\nu_t} \mathbb{Z}.$$
The corresponding valuation ring $V$ is a DVR with maximal ideal generated by $X$. Since 
$$\nu(X) = \nu_t(t), \hspace{1mm} \nu(Y) = \nu_t(p(t)) \geq 1,$$ 
($p(t)$ was scaled so that $t|p(t)$), we see that $\nu$ is centered on $\mathbb{F}_p[X,Y]_{(X,Y)}$. Furthermore, $\nu$ is also trivially centered on its own valuation ring. As $\mathbb{F}_p[[t]]$ dominates $V$ and has residue field $\mathbb F_p$, 
$$\kappa_\nu = \mathbb{F}_p.$$ 
Clearly $\nu$ satisfies equality in (\ref{more-general-Abhyankar-ineq}) with respect to its valuation ring as a center (this is true more generally for any discrete valuation), but not with respect to the center $\mathbb{F}_p[X,Y]_{(X,Y)}$. Note that $\nu$ is \emph{not} an Abhyankar valuation of $\mathbb{F}_p(X,Y)/\mathbb{F}_p$, since
$$\dim_{\mathbb Q}(\mathbb{Q} \otimes_{\mathbb Z} \mathbb{Z}) + \td{\kappa_\nu/\mathbb{F}_p} = 1 \neq  \td{\mathbb{F}_p(X,Y)/\mathbb{F}_p}.$$
Moreover, the valuation ring of $\nu$ is \emph{not} an $F$-finite DVR. This follows from results stated in the next section, but we include the justification here. Indeed, since the maximal ideal $\mathfrak{m}$ of $V$ is principal, by [Section \ref{Summary of F-singularities for valuation rings}, (\ref{assuming-f-finiteness})(i)]
$$\dim_{\kappa^p_\nu}(V/\mathfrak{m}^{[p]}) = p[\kappa_\nu:\kappa^p_\nu] = p \neq  [\mathbb{F}_p(X,Y): \mathbb{F}_p(X,Y)^p],$$
and so $V$ is not $F$-finite by [Section \ref{Summary of F-singularities for valuation rings},  (\ref{general-f-finiteness})]. It turns out that $V$ is also not excellent [Section \ref{Summary of F-singularities for valuation rings}, (\ref{noetherian-Frobenius-splitting})].
\end{example}

Given the example above, we make the following definition.

\begin{definition}
\label{Abhyankar-center}
Let $\nu$ be a valuation centered on a Noetherian local domain $R$. We say $R$ is an \textbf{Abhyankar center} of $\nu$ if
$$\dim_{\mathbb{Q}}(\mathbb{Q} \otimes_{\mathbb{Z}} \Gamma_\nu) + \td{\kappa_\nu/\kappa_R} = \dim(R).$$
\end{definition}  

To summarize our observations, the property of being an Abhyankar valuation of a function field is intrinsic to a valuation, while whether $\nu$ admits Abhyankar centers depends on the centers. However, if additional restrictions are imposed on the class of centers (for instance, if we require centers to be essentially of finite type over $k$), then the property of possessing these more restrictive Abhyankar centers becomes intrinsic to $\nu$.

The interplay between Abhyankar valuations and valuations admitting Abhyankar centers raises the natural question: does (\ref{function-field-defectless}) have an analogue for valuations of fields that are not necessarily function fields? Moreover, can the feature of possessing Abhyankar centers become intrinsic to a valuation if we restrict the class of admissible centers? The next result provides an affirmative answer for a broad class of Noetherian centers.

\begin{theorem}
\label{f-finiteness-general-abhyankar} 
Let $(R, \mathfrak{m}_R, \kappa_R)$ be a Noetherian, $F$-finite local domain of characteristic $p > 0$ and fraction field $K$. Suppose $\nu$ is a non-trivial valuation of $K$ centered on $R$ with value group $\Gamma_\nu$ and valuation ring $(V, \mathfrak{m}_\nu, \kappa_\nu)$.  Then $R$ is an Abhyankar center of $\nu$  if and only if 
$$[\Gamma_\nu:p\Gamma_\nu][\kappa_\nu:\kappa^p_\nu] = [K:K^p].$$
\end{theorem}

Theorem \ref{f-finiteness-general-abhyankar} has some interesting consequences that we illustrate first.

\begin{corollary}
\label{intrinsic-Abhyankar-center}
Let $\nu$ be a valuation of a field $K$ of characteristic $p > 0$. If $\nu$ admits a Noetherian, $F$-finite center which is Abhyankar, then any other Noetherian, $F$-finite center of $\nu$ is also an Abhyankar center of $\nu$. 

In other words, the property of possessing Noetherian, $F$-finite, Abhyankar centers is intrinsic to a valuation.
\end{corollary} 

\begin{proof}[\textbf{Proof of Corollary \ref{intrinsic-Abhyankar-center}}]
The proof follows easily from Theorem \ref{f-finiteness-general-abhyankar} using the observation that the identity $[\Gamma_\nu:p\Gamma_\nu][\kappa_\nu:\kappa^p_\nu] = [K:K^p]$ is independent of the choice of a center.
\end{proof}

\begin{corollary}
\label{excellent-Abhyankar-centers-function-fields}
Let $\nu$ be a valuation of a function field $K/k$ over a perfect field $k$ of characteristic $p > 0$ (it suffices for $k$ to be $F$-finite). The following are equivalent:
\begin{enumerate}
\item[(1)] $\nu$ is an Abyankar valuation of $K/k$.
\item[(2)] $\nu$ admits an Abhyankar center which is an excellent local ring.
\end{enumerate}
\end{corollary}

\begin{proof}[\textbf{Proof of Corollary \ref{excellent-Abhyankar-centers-function-fields}}]
For (1) $\Rightarrow$ (2), any center of the Abhyankar valuation $\nu$ which is essentially of finite type over $k$, hence also excellent, is an Abhyankar center of $\nu$. For the converse, let $R$ be an excellent, Abhyankar center of $\nu$. As $K$ is $F$-finite, $R$ is also $F$-finite. This follows from the fact that since $R^p$ is excellent (it is isomorphic to $R$), its integral closure $S$ in $K$ is module finite over $R^p$, because $K$ is a finite extension of $K^p$. But $R$ is an $R^p$-submodule of $S$, and submodules of finitely generated modules over Noetherian rings are finitely generated. So $R$ is also module finite over $R^p$, that is, $R$ is $F$-finite. Thus $\nu$ satisfies the identity $[\Gamma_\nu:p\Gamma_\nu][\kappa_\nu:\kappa^p_\nu] = [K:K^p]$ by Theorem \ref{f-finiteness-general-abhyankar}, and since $K/k$ is a function field, (\ref{function-field-defectless}) implies that $\nu$ is an Abhyankar valuation of $K/k$.
\end{proof}

We will prove Theorem \ref{f-finiteness-general-abhyankar} by first developing a connection between the inequality
$$\dim_{\mathbb{Q}}(\mathbb{Q} \otimes_{\mathbb{Z}} \Gamma_\nu) + \td{\kappa_\nu/\kappa_R} \leq \dim(R)$$
and the quantities $[\Gamma_\nu:p\Gamma_\nu]$ and $[\kappa_\nu:\kappa^p_\nu]$. This will also shed light on precisely where $F$-finiteness is used in the proof the theorem. 

In order to achieve the above goal, we recall some general facts about torsion-free abelian groups and $F$-finite fields.

\begin{lemma}
\label{few-facts}
Let $p$ be a prime number, $K$ an $F$-finite field of characteristic $p$, and $\Gamma$ a torsion-free abelian group such that $\dim_{\mathbb Q}(\mathbb{Q} \otimes_{\mathbb Z} \Gamma)$ is finite. We have the following:
\begin{enumerate}
\item If $L$ is an algebraic extension of $K$, then
$$[L:L^p] \leq [K:K^p],$$
with equality if $K \subseteq L$ is a finite extension. In particular, $L$ is then also $F$-finite.

\vspace{1mm}

\item If $L$ is field extension of $K$ of transcendence degree $t$, then
$$[L:L^p] \leq p^t[K:K^p],$$
with equality if $L$ is finitely generated over $K$. 

\vspace{1mm}

\item If $s = \dim_{\mathbb Q}(\mathbb{Q} \otimes_{\mathbb Z} \Gamma)$, then
$$[\Gamma:p\Gamma] \leq p^s,$$
with equality if $\Gamma$ is finitely generated.
\end{enumerate}
\end{lemma}

\begin{proof}[Indication of proof of Lemma \ref{few-facts}]

(2) clearly follows from (1), and (3) follows from \cite[Lemma 5.5]{DS16}, with equality obviously holding when $\Gamma$ is finitely generated, since $\Gamma$ is then free. We prove (1) here, which is a minor generalization of \cite[Lemma 5.8]{DS16}.  

To show $[L:L^p] = [K:K^p]$ when $K \subseteq L$ is finite is easy (see \cite[Section 4.6]{DS16}). So suppose $K \subseteq L$ is algebraic, and $[K:K^p] < \infty$. It suffices to show that if $a_1, \dots, a_n \in L$ are linearly independent over $L^p$, then
$$n \leq [K:K^p].$$
Let 
$$\widetilde{L} \coloneqq K(a_1, \dots, a_n).$$
Since $L$ is algebraic over $K$, $\widetilde{L}$ is a finite extension $K$, and so by what we already established,
$$[\widetilde{L}:\widetilde{L}^p] = [K:K^p].$$
On the other hand, since $a_1, \dots, a_n$ are linearly independent over $L^p$, and $\widetilde{L}^p \subseteq L^p$, it follows that $a_1, \dots, a_n$ are also linearly independent over $\widetilde{L}^p$. Thus,
$$n \leq [\widetilde{L}:\widetilde{L}^p] = [K:K^p],$$
as desired.
\end{proof}

Using the previous lemma, we can now relate the ramification index (i.e. $[\Gamma_\nu:p\Gamma_\nu]$) and residue degree (i.e. $[\kappa_\nu:\kappa^p_\nu]$) of the extension of valuations $\nu/\nu^p$ to (\ref{more-general-Abhyankar-ineq}):

\begin{proposition}
\label{general-Abhyankar-centers}
Let $\nu$ be a valuation of a field $K$ of characteristic $p > 0$ with valuation ring $(V, \mathfrak{m}_\nu,\kappa_\nu)$, centered on Noetherian local domain $(R, \mathfrak{m}_R, \kappa_R)$ such that
$$[\kappa_R:\kappa^p_R] < \infty.$$
We have the following:
\begin{enumerate}
\item $[\Gamma_\nu:p\Gamma_\nu][\kappa_\nu:\kappa_\nu^p] \leq p^{\dim(R)}[\kappa_R:\kappa_R^p]$.

\item $R$ is an Abhyankar center of $\nu$ if and only if $[\Gamma_\nu:p\Gamma_\nu][\kappa_\nu:\kappa_\nu^p] = p^{\dim(R)}[\kappa_R:\kappa^p_R].$
\end{enumerate}
\end{proposition}

\begin{proof}[\textbf{Proof of Proposition \ref{general-Abhyankar-centers}}]
Throughout the proof, let
$$s \coloneqq \dim_{\mathbb Q}(\mathbb{Q} \otimes_{\mathbb Z} \Gamma_\nu) \hspace{1mm} \textrm{and} \hspace{1mm} t \coloneqq \td{\kappa_\nu/\kappa_R}.$$

\noindent (1)  Abhyankar's inequality (\ref{more-general-Abhyankar-ineq}) implies 
$$s + t \leq \dim(R).$$ 
In particular, $s$ and $t$ are both finite. Using Lemma \ref{few-facts}(3), we get 
$$[\Gamma_\nu:p\Gamma_\nu] \leq p^s.$$
On the other hand, since $\kappa_R$ is $F$-finite by hypothesis, and $\kappa_\nu$ has transcendence degree $t$ over $\kappa_R$, Lemma \ref{few-facts}(2) shows
$$[\kappa_\nu:\kappa_\nu^p] \leq  p^t[\kappa_R: \kappa_R^p].$$
Thus,
\begin{equation}
\label{ram-res-inequality}
[\Gamma_\nu:p\Gamma_\nu][\kappa_\nu:\kappa_\nu^p] \leq p^{s+t}[\kappa_R:\kappa^p_R] \leq p^{\dim(R)}[\kappa_R:\kappa^p_R].
\end{equation}

\noindent (2) Suppose $R$ is an Abhyankar center of $\nu$, that is, 
$$s + t = \dim(R).$$
By \cite[Theorem 1]{Abh56}, $\Gamma_\nu$ is a free abelian group of rank $s$, and $\kappa_\nu$ is a finitely generated field extension of $\kappa_R$ of transcendence degree $t$. Again using Lemma \ref{few-facts}, we get
$$\textrm{$[\Gamma_\nu:p\Gamma_\nu] = p^s$ and $[\kappa_\nu: \kappa^p_\nu] = p^t[\kappa_R:\kappa^p_R]$},$$ 
and so
$$[\Gamma_\nu:p\Gamma_\nu][\kappa_\nu:\kappa^p_\nu] = p^{s+t}[\kappa_R:\kappa^p_R] = p^{\dim(R)}[\kappa_R:\kappa^p_R],$$
proving the forward implication. 

Conversely, if
$$[\Gamma_\nu:p\Gamma_\nu][\kappa_\nu:\kappa^p_\nu] = p^{\dim(R)}[\kappa_R:\kappa^p_R]$$
then
$$p^{\dim(R)}[\kappa_R:\kappa^p_R] = [\Gamma_\nu:p\Gamma_\nu][\kappa_\nu:\kappa^p_\nu] \leq p^{s+t}[\kappa_R:\kappa^p_R] \leq p^{\dim(R)}[\kappa_R:\kappa^p_R],$$
where the inequalities follow from (\ref{ram-res-inequality}).
Thus, $\dim(R) = s + t$, which by definition means that $R$ is an Abhyankar center of $\nu$. 
\end{proof}

Theorem \ref{f-finiteness-general-abhyankar} now follows readily from Proposition \ref{general-Abhyankar-centers}.

\begin{proof}[\textbf{Proof of Theorem \ref{f-finiteness-general-abhyankar}}]
Suppose $R$ is a Noetherian, $F$-finite, local domain with fraction field $K$. Then 
$$\textrm{$[\kappa_R: \kappa^p_R]$ and $[K:K^p] < \infty$}.$$
In particular, $R$ satisfies the hypotheses of Proposition \ref{general-Abhyankar-centers}, and so Theorem \ref{f-finiteness-general-abhyankar} follows if we can show that
\begin{equation}
\label{estimate-for-K-dim}
[K:K^p] = p^{\dim(R)}[\kappa_R:\kappa^p_R].
\end{equation}
This is a well-known result that is implicit in the proof of \cite[Proposition 2.1]{Kun76}. However, since (\ref{estimate-for-K-dim}) is crucial for our proof, we briefly indicate how it is established. In \cite[Proposition 2.1]{Kun76}, Kunz uses the analogue of Noetherian normalization for complete rings to show that when $R$ is $F$-finite, then for any minimal prime ideal $\mathfrak{P}$ of the $\mathfrak m_R$-adic completion $\widehat{R}$,
$$[K:K^p] = p^{\dim(\widehat{R}/\mathfrak{P})}[\kappa_R:\kappa^p_R].$$
This shows that $\dim(\widehat{R}/\mathfrak{P})$ is independent of $\mathfrak{P}$, or in other words that $\widehat{R}$ is equidimensional. However, since $\mathfrak P$ is minimal, we then have 
$$\dim(\widehat{R}/\mathfrak{P}) = \dim(\widehat{R}) =  \dim(R),$$ 
which confirms (\ref{estimate-for-K-dim}).
\end{proof}

\begin{remarks}
{\*}
\label{Abhyankar-center-remarks}
\begin{enumerate}
\item[(i)] A key point in the proof of Theorem \ref{f-finiteness-general-abhyankar} is that when $(R, \mathfrak m, \kappa)$ is an $F$-finite, Noetherian, local domain of characteristic $p > 0$ and fraction field $K$, then
\begin{equation}
\label{estimate-for-K-dim-2}
[K:K^p] = p^{\dim(R)}[\kappa_R:\kappa^p_R].
\end{equation}
A careful analysis of the proof of \cite[Proposition 2.1]{Kun76} reveals that (\ref{estimate-for-K-dim-2}) holds for any Noetherian, local domain $R$ such that $\Omega^{1}_{R/\mathbb{Z}}$ is a finitely generated $R$-module and the completion $\widehat{R}$ is reduced, that is, if $R$ is analytically unramified. We note that when $R$ is $F$-finite, it satisfies both these properties. Indeed, since $R$ is a finitely generated $R^p$-module, 
$$\Omega^1_{R/\mathbb{Z}} = \Omega^1_{R/R^p}$$
is then a finitely generated $R$-module, and $\widehat{R}$ is reduced by \cite[Lemma 2.4]{Kun69}.

Theorem \ref{f-finiteness-general-abhyankar}, hence Corollary \ref{intrinsic-Abhyankar-center}, clearly hold more generally for the class of Noetherian centers of any $F$-finite field that satisfy (\ref{estimate-for-K-dim-2}).
While such centers are quite common, it is not difficult to construct generically $F$-finite Noetherian local, domains that do not satisfy (\ref{estimate-for-K-dim-2}). For instance, (\ref{estimate-for-K-dim-2}) fails for the non $F$-finite DVR of $\mathbb{F}_p(X,Y)$ constructed in Example \ref{non-excellent-DVR}. In particular, since regular local rings are analytically unramified, our observations imply that the module of absolute K\"ahler differentials of the DVR of Example \ref{non-excellent-DVR} must not be finitely generated.



\vspace{1mm}

\item[(ii)] Theorem \ref{f-finiteness-general-abhyankar} generalizes (\ref{function-field-defectless}). Indeed, if $K/k$ is an $F$-finite function field, then any valuation $\nu$ of $K/k$ admits an $F$-finite, Noetherian center. For instance, by the valuative criterion of properness, $\nu$ is centered on a proper $k$-variety with function field $K$, and the local ring of this variety at the center is $F$-finite. Since we observed that $\nu$ is an Abhyankar valuation if and only if any center of $\nu$ which is essentially of finite type over $k$ is an Abhyankar center, it follows by Theorem \ref{f-finiteness-general-abhyankar} that $\nu$ is Abhyankar precisely when $[\Gamma_\nu:p\Gamma_\nu][\kappa_\nu:\kappa_\nu^p] = [K:K^p]$.  

\vspace{1mm}

\item[(iii)] One can reinterpret Theorem \ref{f-finiteness-general-abhyankar} as saying that an $F$-finite, Noetherian center $R$ of $\nu$ is an Abhyankar center of $\nu$ if and only if the extension of valuations $\nu/\nu^p$ is defectless.

\vspace{1mm}

\item[(iv)] When $[K:K^p] < \infty$, the condition $[\Gamma_\nu:p\Gamma_\nu][\kappa_\nu:\kappa_\nu^p] = [K:K^p]$ does not imply that the valuation $\nu$ admits a Noetherian, $F$-finite center. See Example \ref{valuation-non-perfect-no-F-finite-center} for a counter-example, which shows that counter-examples can be constructed even for valuations of fields that are not perfect.
\vspace{1mm} 

\item[(v)] For a Noetherian domain $R$ with $F$-finite fraction field $K$, the following are all equivalent:
\begin{enumerate}
\item $R$ is $F$-finite.
\item $R$ is excellent.
\item $R$ is a Japanese/N-2 ring.
\item The integral closure of $R^p$ in $K$ is a finite $R^p$-module.
\end{enumerate}
We already saw $\textrm{(b) $\Rightarrow$ (c) $\Rightarrow$ (d) $\Rightarrow$ (a)}$ in the proof of Corollary \ref{excellent-Abhyankar-centers-function-fields}. The hard part is to show $\textrm{(a) $\Rightarrow$ (b)}$, which follows from \cite[Theorem 2.5]{Kun76}.  As a consequence, in Theorem \ref{f-finiteness-general-abhyankar}, the $F$-finiteness assumption on centers can be replaced by excellence, provided we assume the ambient field $K$ is $F$-finite. Hence when the fraction field of a valuation is $F$-finite, the property of admitting excellent Abhyankar centers is intrinsic to the valuation. In particular, this is true for valuations of function fields over perfect ground fields of prime characteristic.

\vspace{1mm}

\item[(vi)] The analogue of (v) is false for valuations of function fields over ground fields of characteristic $0$, that is, whether a valuation admits an excellent Abhyankar center depends on the excellent center. For instance, any DVR of characteristic 0 is automatically excellent \cite[Tag 07QW]{Stacks}, and by imitating the construction of Example \ref{non-excellent-DVR} using the fields $\mathbb{C}(X,Y)$ and $\mathbb{C}((t))$ instead, one can show that there exists a discrete valuation $\nu$ of $\mathbb{C}(X,Y)/\mathbb{C}$ centered on $\mathbb{C}[X,Y]_{(X,Y)}$ such that the latter is not an Abhyankar center of $\nu$ (see \cite[Example 1(iv)]{ELS03} for more details). However, $\nu$ is also trivially centered on its own valuation ring which is an excellent, Abhyankar center of $\nu$. The same example also shows that Corollary \ref{excellent-Abhyankar-centers-function-fields} is false over ground fields of characteristic $0$.

This remark and (v) indicate that excellent rings in characteristic $p > 0$ behave very differently from excellent rings in characteristic $0$, and that the notion of excellence in prime characteristic is more restrictive than in characteristic $0$.
\end{enumerate}
\end{remarks}


\section{Summary of $F$-singularities for valuation rings}
\label{Summary of F-singularities for valuation rings}

\vspace{-1ex}

We summarize all known results on F-singularities of valuation rings, grouping them according to the type of F-singularity they characterize. While most results are proved in \cite{DS16} and the erratum \cite{DS17}, some new results also appear below (with complete proofs). 

For a valuation ring $(V, \mathfrak m, \kappa)$ of a field $K$ of characteristic $p > 0$ with associated valuation $\nu$, singularities of $V$ defined using the Frobenius map are intimately related to properties of the extension of valuations $\nu/\nu^p$, where $\nu^p$ denotes the restriction of $\nu$ to the subfield $K^p$ of $K$. Recall that the valuation ring of $\nu^p$ is $V^p$, and the residue field of $\nu^p$ can be identified with $\kappa^p$. Furthermore, $\nu$ is the unique extension of $\nu^p$ to $K$ (up to equivalence of valuations), and $V$ is the integral closure of $V^p$ in $K$.

In what follows $\Gamma$ or $\Gamma_\nu$ will always denote the value group of a valuation (ring), and $\kappa$ or $\kappa_\nu$ its residue field.



\noindent \textbf{Flatness of Frobenius and F-purity:}

\begin{enumerate}
\item \label{flatness} \cite[Theorem 3.1]{DS16} The Frobenius endomorphism on any valuation ring of prime characteristic is always faithfully flat. Hence a valuation ring of prime characteristic is $F$-pure, and so close to being Frobenius split. 

\begin{remark}
It is not difficult to construct valuation rings for which Frobenius is pure but not split. For example, the non $F$-finite DVR of Example \ref{non-excellent-DVR} is not Frobenius split, because any Frobenius split Noetherian domain with $F$-finite fraction field has to $F$-finite \cite[Theorem 4.2.1]{DS16} (see also (\ref{noetherian-Frobenius-splitting}) below).
\end{remark}

\end{enumerate}

\noindent \textbf{$F$-finiteness in general:}

Let $(V, \mathfrak m, \kappa)$ be a valuation ring of a field $K$, with associated valuation $\nu$. A necessary condition for $V$ to be $F$-finite is that $[K:K^p] < \infty$, that is, $K$ is $F$-finite. So we implicitly assume in our discussion of $F$-finiteness of $V$ that $K$ is $F$-finite to begin with. Note $F$-finiteness of $K$ also implies $$[\kappa:\kappa^p]< \infty,$$ 
that is, the residue field is always $F$-finite. This follows by observing that $[\kappa:\kappa^p]$ is the \emph{residue degree} of the extension of valuations $\nu/\nu^p$ and then using \cite[VI, $\mathsection 8.1$, Lemma 2]{Bou89}.

\begin{enumerate}

\setcounter{enumi}{1}\item \label{general-f-finiteness}  The following are equivalent:
\begin{enumerate}
\item[(a)] $V$ is $F$-finite.
\item[(b)] $V$ is a free $V^p$-module of rank $[K:K^p]$. 
\item[(c)] $\dim_{\kappa^p}(V/\mathfrak{m}^{[p]}) = [K:K^p]$.
\end{enumerate}

\vspace{1mm}

\begin{proof}[\textbf{Proof of (2)}]
The equivalence of (a) and (b) is shown in \cite[Theorem 4.1.1]{DS16}. Although used in the proof of \cite[Erratum, revised Corollary 4.3.2]{DS17}, the equivalence of (c) to (a) and (b) is not explicitly stated in \cite{DS16, DS17}. Thus, we include a complete proof here. 

Let $n \coloneqq [K:K^p]$. We show (b) and (c) are equivalent. Suppose $V$ is a free module of rank $n$ over the subring $V^p$, which is a valuation ring of $K^p$. If $\eta$ is the maximal ideal of $V^p$, we see that 
$$V/\mathfrak{m}^{[p]} \cong V \otimes_{V^p} V^p/\eta$$
is a free $\kappa^p = V^p/\eta$-module of rank $n$, which proves (b) $\Rightarrow$ (c). For the converse, suppose $\dim_{\kappa^p}(V/\mathfrak{m}^{[p]}) = [K:K^p] =  n$. Choose $x_1, \dots, x_n \in V$ such that the images of $x_i$ in $V/\mathfrak{m}^{[p]}$ form a $\kappa^p$-basis, and let
$$L \coloneqq V^px_1 + \dots + V^px_n.$$ 
Note $L$ is a finitely generated, torsion free $V^p$-module, hence free over $V^p$ since finitely generated torsion-free modules over valuation rings are free. To prove (b), it suffices to show that
$$L = V.$$ 
The rank of $L$ equals $\dim_{\kappa^p} L/\eta L$, and it is easy to see that the images of $x_1, \dots, x_n$ in $L/\eta L$ form a $\kappa^p$-basis of $L/\eta L$. Thus, $L$ is a free $V^p$-module of rank $n$, and so $\{x_1, \dots, x_n\}$ is a $V^p$-basis of $L$. 

Observe that $\{x_1, \dots, x_n\}$ is also linearly independent over $K^p$, and since $[K:K^p] = n$, this means that $\{x_1, \dots, x_n\}$ is a $K^p$-basis of $K$. Let $s \in V$ be a non-zero element, and $r_1, \dots, r_n \in K^p$ such that
$$s = r_1x_1 + \dots + r_nx_n.$$
Clearly $V = L$, if we can prove that all the $r_i$ are elements of $V^p$. By renumbering the $x_i$, and using the fact that $V^p$ is a valuation ring of $K^p$, we may assume without loss of generality that $r_1 \neq 0$ and 
$$r_ir_1^{-1} \in V^p,$$ 
for all $i \geq 2$. If $r_1 \in V^p$, then all the $r_i$ are already in $V^p$. If not, then $r_1^{-1}$ is an element of the maximal ideal, $\eta$, of $V^p$. Thus,
$$r_1^{-1}s = x_1 + r_2r_1^{-1}x_2 + \dots + r_nr_1^{-1}x_n,$$
which contradicts $\kappa^p$-linear independence of the images of $x_1, \dots, x_n$ in $V/\eta V = V/\mathfrak{m}^{[p]}$. Hence all the $r_i$ are elements of $V^p$, and we are done.
\end{proof}

\vspace{2mm}

\item \label{assuming-f-finiteness} 
\begin{enumerate}

\item[(i)] \cite[Erratum, Lemma 2.2]{DS17} 
For $(V, \mathfrak m, \kappa)$ and $K$ as above,
\vspace{-1ex} 
\[   
 \dim_{\kappa^p}(V/\mathfrak{m}^{[p]}) = 
     \begin{cases}
       [\kappa: \kappa^p] &\quad\text{if $\mathfrak m$ is not finitely generated,}\\
       p[\kappa:\kappa^p] &\quad\text{if $\mathfrak m$ is finitely generated.} \\ 
     \end{cases}
\]

\begin{proof}[\textbf{Sketch of proof of (i)}] 
 The assertion follows from the short exact sequence of $\kappa^p$-vector spaces
$$0 \rightarrow \mathfrak{m}/\mathfrak{m}^{[p]} \rightarrow V/\mathfrak{m}^{[p]} \rightarrow \kappa \rightarrow 0,$$
with the additional observations that when $\mathfrak{m}$ is not finitely generated, $\mathfrak{m}^{[p]} = \mathfrak m$ \cite[Lemma 2.1]{DS17}, and when $\mathfrak{m}$ is finitely generated, $\mathfrak{m}$ is principal, so that $\dim_{\kappa^p}(\mathfrak{m}/\mathfrak{m}^{[p]}) = (p-1)[\kappa:\kappa^p]$.
\end{proof}



\vspace{1.5mm}
\item[(ii)] For an \emph{F-finite}, valuation ring $V$ with value group $\Gamma \neq 0$:
\begin{itemize}
\item[(a)] $[K:K^p] = [\Gamma:p\Gamma][\kappa:\kappa^p] =  \dim_{\kappa^p}(V/\mathfrak{m}^{[p]})$ (\cite[corrected Thm 4.3.1]{DS17} and (\ref{general-f-finiteness})).

\item[(b)] The value group $\Gamma$ satisfies $[\Gamma: p\Gamma] = 1$ or $[\Gamma:p\Gamma] = p$. 

\item[(c)] If the maximal ideal of $V$ is \emph{not} finitely generated, then $\Gamma$ is $p$-divisible. 


\item[(d)] If $\Gamma$ is finitely generated, then $V$ is a DVR.  
\end{itemize}

\begin{remark}
\label{main-error}
The \textbf{error} in \cite{DS16} arose from the incorrect assertion that  $[\Gamma:p\Gamma][\kappa:\kappa^p] = [K:K^p] \Rightarrow V$ is F-finite (although the assertion is true when $V$ is a DVR by (3)(iv) below). 
\end{remark}

\vspace{1.5mm}

\item[(iii)] If $[K: K^p] = [\kappa:\kappa^p]$, then $V$ is F-finite \cite[revised Corollary 4.3.2]{DS17}.

\vspace{1.5mm}

\item[(iv)] If $V$ is a DVR, then $V$ is $F$-finite if and only if $[\Gamma:p\Gamma][\kappa:\kappa^p] = [K:K^p]$.

\vspace{1mm}

\begin{proof}[\textbf{Proof of (iv)}]
If $V$ is $F$-finite, then $[\Gamma:p\Gamma][\kappa:\kappa^p] = [K:K^p]$ by (ii)(a) above. For the converse, we have 
$$[K:K^p] = [\Gamma:p\Gamma][\kappa:\kappa^p] = [\mathbb{Z}:p\mathbb Z][\kappa:\kappa^p] = p[\kappa:\kappa^p] = \dim_{\kappa^p}(V/\mathfrak{m}^{[p]}),$$
where we use (\ref{assuming-f-finiteness})(i) for the final equality. So $V$ is $F$-finite by (\ref{general-f-finiteness}).
\end{proof}

\vspace{1mm}

\begin{remark}
The proof of (iv) shows more generally that if $V$ is a valuation ring with principal maximal ideal $\mathfrak{m}$ and such that $[\Gamma:p\Gamma] = p$, then $[\Gamma:p\Gamma][\kappa:\kappa^p] = [K:K^p]$ implies $V$ is $F$-finite.
\end{remark}

\vspace{1mm}
\end{enumerate}

\end{enumerate}

\noindent \textbf{$F$-finiteness and finite field extensions:}

\begin{enumerate}
\setcounter{enumi}{3} \item \label{F-finiteness-preserved-finite-ext}\cite[Section 3]{DS17}
Let $K \subseteq L$ be a finite extension of $F$-finite fields. Let $\nu$ be a valuation on $K$ and $w$ an extension of $\nu$ to $L$. Then the valuation ring of $\nu$ is $F$-finite if and only if the valuation ring of $w$ is $F$-finite.
\end{enumerate}




\vspace{1.5mm}

\noindent \textbf{$F$-finiteness in function fields:}

Let $K$ be a finitely generated field extension of an F-finite field $k$.

\begin{enumerate}
\setcounter{enumi}{4} \item \label{abhyankar-defectless} \cite[Theorem 1.1]{DS17}  A valuation $\nu$ of $K/k$ is Abhyankar if and only if 
$$[\Gamma_\nu: p\Gamma_\nu][\kappa_\nu:\kappa^p_\nu] = [K:K^p].$$
See also Theorem \ref{f-finiteness-general-abhyankar} for a generalization of the above result.

\vspace{1mm}

\item \label{f-finiteness-divisorial}\cite[Erratum, Theorem 0.1]{DS17} For a non-trivial valuation $\nu$ of $K/k$, the associated valuation ring $V$ is $F$-finite if and only if $\nu$ is divisorial.

\begin{proof}[\textbf{Sketch of proof of (\ref{f-finiteness-divisorial})}]
The non-trivial implication is that when $V$ is $F$-finite, its valuation $\nu$ is divisorial. However, (\ref{assuming-f-finiteness})(ii)(a) and (\ref{abhyankar-defectless}) imply that $\nu$ is Abhyankar, and so has a finitely generated value group. Using (\ref{assuming-f-finiteness})(ii)(d) one then concludes $V$ is a DVR. Now a classical result of Zariski shows that any Abhyankar DVR is divisorial \cite[VI, $\mathsection 14$, Theorem 31]{SZ60}.
\end{proof}

\begin{remarks}
{\*}
\label{perfect-valuation-rings}
\begin{enumerate}
\item[(i)] \cite[Theorem 5.1]{DS16} erroneously states that an Abhyankar valuation ring of $K/k$ is F-finite. The justification for why this is wrong is given in Remark \ref{main-error}. For a counter-example, the valuation ring $V$ of the lexicographical valuation on $\mathbb{F}_p(X, Y)/\mathbb{F}_p$ with value group $\mathbb{Z}^{\oplus 2}$ (ordered lexicographically) is Abhyankar (see Example \ref{some-examples}(a)), but not $F$-finite because
$$\dim_{\kappa^p}(V/\mathfrak{m}^{[p]}) = p[\kappa:\kappa^p] = p[\mathbb{F}_p:\mathbb{F}_p] = p \neq  [\mathbb{F}_p(X, Y): \mathbb{F}_p(X, Y)^p].$$
Here the first equality holds by (\ref{assuming-f-finiteness})(i) because the maximal ideal of $V$ is principal, with $Y$ being a generator. The second equality holds because the residue field of $V$ is $\mathbb F_p$.  

\vspace{1.5mm}

\item [(ii)] When not in a function field, it is easy to construct non-Noetherian $F$-finite valuation rings. For example, the perfection $\mathbb{F}_p[[t^{1/p^{\infty}}]] \coloneqq \bigcup_{e \in \mathbb{N}} \mathbb{F}_p[[t^{1/p^e}]]$ of the power series ring $\mathbb{F}_p[[t]]$ is a non-Noetherian, $F$-finite valuation ring of its fraction field $\mathbb{F}_p((t^{1/p^{\infty}}))$. More generally, a non-trivial valuation ring of any perfect field of prime characteristic is not Noetherian, but $F$-finite because Frobenius is an isomorphism for such a ring. Rings of prime characteristic for which Frobenius is an isomorphism are called \emph{perfect rings}. Such rings have been extensively investigated of late since finding applications in Scholze's work on perfectoid spaces (see \cite{Sch12}, \cite{BS16}). While perfect rings are trivially $F$-finite, there exist non-Noetherian, $F$-finite valuation rings that are \emph{not} perfect. 

\vspace{1mm}

Suppose $L$ is a perfect field of prime characteristic equipped with a non-trivial valuation $\nu$ with value group $\Gamma_\nu$. For instance $L$ can be a perfectoid field, or the algebraic closure of a field which has non-trivial valuations. Then the residue field $\kappa_\nu$ of the associated valuation ring is also perfect. Now consider the group 
$$\Gamma' \coloneqq \Gamma_\nu \oplus \mathbb{Z}$$ 
ordered lexicographically, and the field $L(X)$, where $X$ is an indeterminate. There exists a \emph{unique} extension $w$ of the valuation $\nu$ to $L(X)$ with value group $\Gamma'$ such that
for any polynomial $f = \sum_{i=0}^n a_iX^i$ in $L[X]$, we have
$$w(f) = \inf\{(\nu(a_i), i): i = 0, \dots, n\}.$$
The residue field $\kappa_w$ of $w$ equals the residue field $\kappa_\nu$ \cite[VI, $\mathsection 10.1$, Proposition 1]{Bou89}, hence is also perfect. Also, $\Gamma'$ has a smallest element $> 0$ in the lex order, namely $(0,1)$. Thus, if $(R_w, \mathfrak{m}_w)$ is the valuation ring of $w$, the maximal ideal $\mathfrak{m}_w$ is principal, and in fact generated by $X$. Using (3)(i), we see that
\begin{equation}
\label{equation1}
\dim_{\kappa^p_w}(R_w/\mathfrak{m}^{[p]}_w) = p[\kappa_w:\kappa^p_w] = p = [L(X):L(X)^p].
\end{equation}
Then $R_w$ is $F$-finite by (\ref{general-f-finiteness}), not Noetherian because $\Gamma' = \Gamma_\nu \oplus \mathbb{Z}$ has rational rank at least $2$, and {not} perfect because the field $L(X)$ is not perfect.

\vspace{1mm}

\item[(iii)] Curiously, if instead of taking $\Gamma' = \Gamma_\nu \oplus \mathbb{Z}$ ordered lexicographically we take $\Gamma' = \mathbb{Z} \oplus \Gamma_\nu$ ordered lexicographically in the above construction, the resulting  extension $w$ of $\nu$ to $L(X)$ (with obvious modifications to the definition of $w$) does not have an $F$-finite valuation ring $R_w$. Indeed, then the maximal ideal of $R_w$ is not finitely generated, while the residue field $\kappa_w$ still coincides with $\kappa_\nu$, which is perfect. Thus $\dim_{\kappa^p_w}(R_w/\mathfrak{m}^{[p]}_w) = [\kappa_w:\kappa^p_w] = 1 \neq [L(X):L(X)^p]$.
\end{enumerate} 

\end{remarks}

\end{enumerate}

\noindent \textbf{$F$-finiteness and valuations centered on Noetherian domains:}

Using Theorem \ref{f-finiteness-general-abhyankar}, one can generalize (\ref{f-finiteness-divisorial}) to a non function field setting as follows.

\begin{enumerate}
\setcounter{enumi}{6}
\item \label{F-finite-aluation-centered-Noetherian-domain} Let $\nu$ be a non-trivial valuation on $K$ centered on an $F$-finite, Noetherian local domain $R$. Then the valuation ring $V$ of $\nu$ is $F$-finite if and only if $V$ is a DVR and $R$ is an Abhyankar center of $\nu$.

\begin{proof}[\textbf{Proof of (\ref{F-finite-aluation-centered-Noetherian-domain})}] 
For the forward implication, if $V$ is F-finite, then $[\Gamma_\nu:p\Gamma_\nu][\kappa_\nu:\kappa^p_\nu] = [K:K^p]$ holds automatically (see (\ref{assuming-f-finiteness})(ii)), and then Theorem \ref{f-finiteness-general-abhyankar} implies that $R$ is an Abhyankar center of $\nu$. In particular, $\Gamma_\nu$ is a non-trivial finitely generated abelian group, and so (\ref{assuming-f-finiteness})(ii)(d) $\Rightarrow$ $V$ is a DVR. This proves the forward implication. Conversely, if $R$ is an Abhyankar center of $\nu$, then Theorem \ref{f-finiteness-general-abhyankar} again implies
$$[\Gamma_\nu:p\Gamma_\nu][\kappa_\nu:\kappa_\nu^p] = [K:K^p].$$
Since $V$ is also a DVR by hypothesis, it is $F$-finite by (\ref{assuming-f-finiteness})(iv).
\end{proof}

\end{enumerate}

\vspace{1mm}

The above result has the following interesting consequence that we would like to highlight separately.

\begin{enumerate}
\setcounter{enumi}{7}

\item \label{non-Noetherian-$F$-finite-rings}
Suppose $\nu$ is a valuation of an $F$-finite field $K$ with valuation ring $V$ that satisfies either of the following conditions:
\begin{enumerate}
\item[(a)] $V$ is $F$-finite, but not Noetherian.
\item[(b)] $\dim(V) > s$, where $[K:K^p] = p^s$.
\end{enumerate}
Then $\nu$ is not centered on any excellent local domain whose fraction field is $K$.

\begin{proof}[\textbf{Proof of (\ref{non-Noetherian-$F$-finite-rings})}]
Since $K$ is $F$-finite, a Noetherian domain with fraction field $K$ is excellent if and only if it is $F$-finite (see Remark \ref{Abhyankar-center-remarks}(v)). Thus, it suffices to show that $\nu$ is not centered on any $F$-finite, Noetherian local domain if it satisfies (a) or (b).

Suppose $\nu$ satisfies (a). As $V$ is not Noetherian, (\ref{F-finite-aluation-centered-Noetherian-domain}) implies that $\nu$ cannot be centered on any $F$-finite, Noetherian local domain with fraction field $K$. 

If $R$ is an $F$-finite, Noetherian local ring with fraction field $K$, then recall that we have the identity
$$p^{\dim(R)}[\kappa_R:\kappa^p_R] = [K:K^p].$$
In particular, this means $\dim(R) \leq s$, where $s$ is as above. Now if $\nu$ is centered on $R$, then Abhyankar's inequality (\ref{more-general-Abhyankar-ineq}) shows in particular that
$$\dim_{\mathbb Q}(\mathbb Q \otimes_{\mathbb Z} \Gamma_\nu) \leq \dim(R) \leq s.$$
However, it is well-known that the Krull dimension of $V$ is at most $\dim_{\mathbb Q}(\mathbb Q \otimes_{\mathbb Z} \Gamma_\nu)$ \cite[$\mathsection$ 10.2, Corollary to Proposition 3]{Bou89}. Then $\dim(V) \leq s$, which contradicts the hypothesis of (b). Hence $\nu$ cannot be centered on any $F$-finite, Noetherian domain with fraction field $K$.
\end{proof}


\begin{example}
\label{valuation-non-perfect-no-F-finite-center}
Let $w$ be the valuation of $L(X)$ (where $L$ is a perfect field) constructed in Remark \ref{perfect-valuation-rings}(ii). The valuation ring $R_w$ satisfies conditions (a) and (b) of (\ref{non-Noetherian-$F$-finite-rings}). We have already observed that $R_w$ satisfies (a). To see that $R_w$ satisfies (b), note that the value group of $w$ has a proper, non-trivial isolated/convex subgroup. Thus $R_w$ has Krull dimension at least 2 \cite[$\mathsection$4.5]{Bou89}, while $[L(X):L(X)^p] = p$.  

Although $R_w$ is a valuation ring of a function field, it does not contain the ground field $L$. So even though $w/w^p$ is defectless, this example does not contradict (\ref{abhyankar-defectless}), or the problem of local uniformization in prime characteristic. 
\end{example}


\begin{remark}
\label{function-vs-arbitrary-field}
If $K/k$ is an $F$-finite function field, and $\nu$ is a valuation of $K/k$ with valuation ring $V$, then (\ref{abhyankar-defectless}) shows that $V$ cannot satisfy (\ref{non-Noetherian-$F$-finite-rings})(a), while (\ref{Abh-ineq}) shows that $V$ cannot satisfy (\ref{non-Noetherian-$F$-finite-rings})(b). Thus, the pathologies of (\ref{non-Noetherian-$F$-finite-rings}) do not arise for valuations of function fields that are trivial on the ground field.
\end{remark}
\end{enumerate}

\vspace{2mm}

\noindent \textbf{Frobenius splitting:}

\begin{enumerate}

\setcounter{enumi}{8}\item \label{f-finite-frobenius-splitting} \cite[Corollary 4.1.2]{DS16} Any $F$-finite valuation ring is Frobenius split. This follows from (\ref{general-f-finiteness}) because $F_*V$ is then a free $V$-module.

\begin{remark}
(\ref{f-finite-frobenius-splitting}) is a special case of a general phenomenon of valuation rings which is independent of the characteristic of the ring. Recall, that a ring $R$ is called a \emph{splinter} if any module finite ring extension
$$\varphi: R \rightarrow S$$
has an $R$-linear left inverse. For example, from recent work of Andr\'e \cite{And18} (see also \cite{Bha18, HM18}) and earlier work of Hochster \cite{Hoc73}, it follows that any regular ring is a splinter. We want to show that valuation rings are also splinters. So let $V$ be a valuation ring (of any characteristic), and 
$$\varphi: V \rightarrow S$$
a module finite ring extension. Choose a prime ideal $P$ of $S$ that lies over the zero ideal of $V$. Then the composition
$$V \xrightarrow{\varphi} S \twoheadrightarrow S/P$$
is also a module finite ring extension, and it suffices to show this composite extension splits. Thus we may assume $S$ is a domain, which makes $S$ a finitely generated torsion-free $V$ module. But finitely generated torsion free modules over valuation rings are free. Nakayama's lemma implies there exists a basis  of $S$ over $V$ containing $1$, and then one can easily construct many splittings of $\varphi$ with respect to such a basis.
\end{remark}

\vspace{1.5mm}

\item \label{noetherian-Frobenius-splitting} \cite[Corollary 4.2.2]{DS16} The following are equivalent for a \emph{Noetherian} valuation ring $(V, \mathfrak m, \kappa)$ with F-finite fraction field $K$:

$\bullet$ $V$ is Frobenius split.

$\bullet$ $V$ is $F$-finite.

$\bullet$ $V$ is excellent.

$\bullet$ $\dim_{\kappa^p}(V/\mathfrak{m}^p) = [K:K^p]$.

$\bullet$ $[\Gamma:p\Gamma][\kappa:\kappa^p] = [K:K^p]$.

\vspace{1.5mm}

\item \label{Frobenius-splitting-DVR-Abhy-center} Any DVR that admits a Noetherian, $F$-finite, Abhyankar center is Frobenius split.

\begin{proof}[\textbf{Proof of \ref{Frobenius-splitting-DVR-Abhy-center}}] $V$ is $F$-finite by (\ref{F-finite-aluation-centered-Noetherian-domain}), hence Frobenius split by (\ref{f-finite-frobenius-splitting}).
\end{proof}

\vspace{1.5mm}

\item \label{complete-DVR} Any complete DVR of prime characteristic is Frobenius split.

\begin{proof}[\textbf{Proof of \ref{complete-DVR}}] Let $V$ be a complete DVR of prime characteristic with residue field $\kappa$. Since $V$ is equicharacteristic, by Cohen's Structure Theorem for complete rings, 
$$V \cong \kappa[[t]],$$
as rings, and clearly $\kappa^p[[t^p]]$ is a direct summand of $\kappa[[t]]$, even when $[\kappa:\kappa^p]$ is not finite.
\end{proof}

\vspace{1.5mm}

\item \label{Frobenius-splitting-Abhyankar-vals} [Section \ref{Proof of Theorem Theorem-A}] For an Abhyankar valuation ring $V$ of an $F$-finite function field $K/k$, if the residue field $\kappa$ is separable over $k$, then $V$ is Frobenius split. In particular Abhyankar valuation rings over perfect ground fields of prime characteristic are always Frobenius split.

\vspace{1.5mm}

\item \label{totally-unramified} We have seen in (\ref{abhyankar-defectless}) that Abhyankar valuations in $F$-finite function fields are characterized as those valuations $\nu$ such that $\nu/\nu^p$ is defectless, that is, $[\Gamma_\nu:p\Gamma_\nu][\kappa_\nu:\kappa^p_\nu] = [K:K^p]$. At the opposite extreme is a valuation $\nu$ such that the extension $\nu/\nu^p$ is \emph{totally unramified}, which means that
\begin{equation}
\label{total-ram}
[\Gamma_\nu:p\Gamma_\nu][\kappa_\nu:\kappa^p_\nu] = 1.
\end{equation}
We want to show that when $\nu/\nu^p$ is totally unramified and $K$ is not perfect, the valuation ring $V$ of $\nu$ \emph{cannot} be Frobenius split\footnote{The author thanks Ray Heitmann for showing him a proof of (\ref{totally-unramified}) in the special case of $\mathbb Q$-valuations of $\mathbb{F}_p(X,Y)$. Our argument above is a generalization of Heitmann's argument to all totally unramified extensions $\nu/\nu^p$.}. Note that (\ref{total-ram}) implies $\Gamma_\nu = p\Gamma_\nu$ and $\kappa_\nu = \kappa^p_\nu$, that is, the value group is $p$-divisible and the residue field is perfect. The $p$-divisibility of $\Gamma_\nu$ shows that
$$\mathfrak{m} = \mathfrak{m}^{[p]}.$$
Then any Frobenius splitting $\varphi: V \rightarrow V^p$ maps the maximal ideal $\mathfrak{m}$ of $V$ into the maximal ideal of $V^p$, thereby inducing a Frobenius splitting of residue fields 
$$\widetilde{\varphi}: \kappa_\nu \rightarrow \kappa^p_\nu.$$
However, $\kappa_\nu$ is perfect, so that $\widetilde{\varphi}$ is just the identity map. Since $K$ is not perfect, $\varphi$ has a non-trivial kernel, that is, some non-zero $x \in V$ gets mapped to $0$. By $p$-divisibility, one can write $x = uy^p$, for a unit $u$ in $V$, and $y \neq 0$. Then $0 = \varphi(x) = y^p\varphi(u)$, which shows that $\varphi(u) = 0$. But this contradicts injectivity of $\widetilde\varphi$, proving that no Frobenius splitting of $V$ exists.

\begin{example}
\label{ex:non-F-split}
Valuations $\nu$ such that $\nu/\nu^p$ is totally unramified can occur even in function fields. For example, one can construct a $\mathbb{Q}$-valued valuation of $\mathbb{F}_p(X,Y)/\mathbb{F}_p$ (see \cite[Example 8]{Vaq06}). We claim that $\nu/\nu^p$ is then totally unramified. Indeed, since $\mathbb{Q}$ is not a finitely generated abelian group, $\nu$ is not an Abhyankar valuation of $\mathbb{F}_p(X,Y)/\mathbb{F}_p$. It follows by Abhyankar's inequality (\ref{Abh-ineq}) that if $\kappa_\nu$ is the residue field of $\nu$, then $\td{\kappa_\nu/\mathbb{F}_p} = 0$, that is, $\kappa_\nu$ is an algebraic extension of $\mathbb{F}_p$. Consequently, $\kappa_\nu$ is a perfect field, and so $[\kappa_\nu:\kappa_\nu^p] = 1$. Furthermore, $\mathbb{Q}$ is $p$-divisible. This shows that the equality in (\ref{total-ram}) holds, that is, $\nu/\nu^p$ is totally unramified.
\end{example}


\item \label{spherical-complete} Let $V$ be a valuation ring of Krull dimension $1$ and fraction field $K$. Since the value group of $V$ is an ordered subgroup of $\mathbb{R}$, the induced valuation $\nu$ on $K$ gives a multiplicative norm $|\cdot |_\nu: K \rightarrow \mathbb{R}_{\geq 0}$ via the assignment $|x|_\nu = e^{-\nu(x)}$ for $x \in K-\{0\}$ and $|0|_\nu = 0$. The norm $|\cdot |_\nu$ satisfies the multiplicative analogue of the axioms of a valuation:
\begin{enumerate}
	\item $|xy|_\nu = |x|_\nu|y|_\nu$
	\item $|x+y|_\nu \leq \max\{|x|_\nu, |y|_\nu\}$.
\end{enumerate}
In other words, $|\cdot |_\nu$ is an \emph{ultrametric} on $K$. Observe that the valuation ring $V$ of $K$ is then the subring of elements $x \in K$ such that $|x|_\nu \leq 1$. Given $x \in K$ and $r \in \mathbb{R}_{\geq 0}$, the \emph{closed ball} $B(x,r)$ of radius $r$ centered at $x$ is defined the usual way. A consequence of $|\cdot |_\nu$ being ultrametric is that if $r > 0$, then any closed ball is also open. We say that the field $K$ is \emph{spherically complete} with respect to the norm $|\cdot |_\nu$ if any nested sequence of non-empty closed balls $B_1 \supset B_2 \supset B_3 \supset \dots$ has a non-empty intersection. Note that here we do not require the centers of the $B_i$ to be the same, in which case the assertion that the intersection is non-empty is clear. If $K$ is spherically complete with respect to $|\cdot|_\nu$ then $K$ is also complete with respect to this norm. One can view the notion of a spherically complete field as a generalization of a field equipped with a complete discrete valuation, for a field of the latter type is always spherically complete \cite[Theorem 1.2.13]{GPS10}

We now claim that if $V$ has prime characteristic $p > 0$ and $(K, |\cdot |_\nu)$ is spherically complete with respect to $|\cdot|_\nu$, then $V$ is Frobenius split\footnote{Eric Canton, Takumi Murayama, Matthew Stevenson and the author realized (\ref{spherical-complete}) while working on a different project. He is grateful to them for allowing him to reproduce the result in this paper.}. This result may be viewed as a generalization of Frobenius splitting of complete discrete valuation rings (\ref{complete-DVR}). Let $|\cdot |_{\nu^p}$ be the corresponding norm on the subfield $K^p$ induced by $\nu^p$. Then $(K^p, |\cdot|_{\nu^p})$ is also spherically complete. Viewing $K$ as a vector space over $K^p$, one can verify that the norm $|\cdot |_\nu$ on $K$ gives $K$ the structure of a normed space over $K^p$ (i.e. the norm on $K$ is compatible with that on $K^p$). Since $K^p$ is spherically complete, the analytic form of the Hahn-Banach theorem over spherically complete fields \cite[Theorem 4.1.1]{GPS10} shows that the identity map
\[
\id_{K^p} : K^p \rightarrow K^p
\]
extends to a $K^p$-linear map $f: K \rightarrow K^p$ such that for all $x \in K$, 
$$|f(x)|_{\nu^p} \leq |x|_\nu.$$
Since the identity map sends $1 \mapsto 1$, it follows that $f$ also sends $1 \mapsto 1$, that is, $f$ is a Frobenius splitting of $K^p$. On the other hand, for any $x \in V$, $|f(x)|_{\nu^p} \leq |x|_\nu \leq 1$, that is, if $x \in V$, then $f(x) \in K^p \cap V = V^p$. Thus, restricting $f$ to $V$ gives a Frobenius splitting of $V$.

\end{enumerate}

\vspace{2mm}

\noindent\textbf{$F$-regularity:}

Two notions of $F$-regularity were introduced in \cite{DS16}, generalizing strong $F$-regularity to a non-Noetherian and non $F$-finite setting-- \emph{split $F$-regularity} \cite[Definition 6.6.1]{DS16} and \emph{$F$-pure regularity} \cite[Definiton 6.1.1]{DS16}. Split $F$-regularity just drops the $F$-finite and Noetherian hypotheses from the definition of strong $F$-regularity, while $F$-pure regularity replaces splitting of certain maps by purity and seems to be the better notion for rings that are not $F$-finite. Split $F$-regularity $\Rightarrow$ $F$-pure regularity, but the converse is false. Indeed, any non-excellent DVR with an $F$-finite fraction field will be $F$-pure regular but not split $F$-regular. In particular, the DVR of Example \ref{non-excellent-DVR} is not excellent, so not split $F$-regular.

\begin{enumerate}
\setcounter{enumi}{15} \item \label{F-pure-regularity} \cite[Thm 6.5.1 and Cor. 6.5.4]{DS16} For a valuation ring $V$ of prime characteristic, $V$ is $F$-pure regular if and only if it is Noetherian.

\vspace{2mm}

\item \label{split-F-regularity} \cite[Corollary 6.6.3]{DS16} Let $V$ be a valuation ring whose fraction field is $F$-finite. The following are equivalent (see also (\ref{noetherian-Frobenius-splitting})):

$\bullet$ $V$ is split $F$-regular.

$\bullet$ $V$ is Noetherian and $F$-finite.

$\bullet$ $V$ is excellent.

$\bullet$ $V$ is Noetherian and Frobenius split.

$\bullet$ If $V$ is a valuation ring of a function field, then $V$ is divisorial.
\end{enumerate}

\begin{remark}
(\ref{F-pure-regularity}) and (\ref{split-F-regularity}) indicate that $F$-regularity is perhaps a useful notion of singularity only for Noetherian rings. 
\end{remark}

\noindent \textbf{Open Questions:} Just as is the case in geometry, our results indicate that Frobenius splitting is the most mysterious F-singularity for valuation rings with many basic open questions.

The proof of Frobenius splitting of Abhyankar valuation rings of function fields over $F$-finite ground fields uses the local monomialization result of Knaf and Kuhlmann (Theorem \ref{local-uniformization}), hence also the hypothesis that the residue field of the valuation ring is separable over the ground field. However, it is probably the case that any Abhyankar valuation ring of an $F$-finite function field is Frobenius split, and this will from our proof if one can remove the separability hypothesis from Theorem \ref{local-uniformization}. Moreover, a natural question is if one can generalize our result on Frobenius splitting of Abhyankar valuations to valuations, not necessarily of function fields, that admit $F$-finite, Noetherian, Abhyankar centers satisfying `mild' singularities such as $F$-regularity. For example, (\ref{F-finite-aluation-centered-Noetherian-domain}) shows that a discrete valuation admitting a Noetherian, $F$-finite, Abhyankar center is Frobenius split.

$F$-singularities of a valuation ring $V$ are intimately related to basic properties of the corresponding extension of valuations $\nu/\nu^p$. For example, (\ref{abhyankar-defectless}) shows that at least for function fields over perfect ground fields, when $\nu/\nu^p$ is defectless, $\nu$ is Abhyankar and its valuation ring is Frobenius split. On the other hand, when $\nu/\nu^p$ has maximal defect, that is when $\nu/\nu^p$ is totally unramified, (\ref{totally-unramified}) shows that $V$ cannot be Frobenius split unless it is a perfect ring. However, Frobenius splitting of $V$ remains mysterious when the defect of $\nu/\nu^p$ is not one two possible extremes. 

For instance, suppose $K/k$ is an $F$-finite function field. Is there a non-Abhyankar valuation of $K/k$ whose valuation ring is Frobenius split? We can use (\ref{split-F-regularity}) to conclude that such valuations, if they exist, cannot be discrete. On the other hand, Example \ref{valuation-non-perfect-no-F-finite-center} shows that if we relax the condition that the valuation is trivial on the ground field, then there exists a valuation $w$ of $K$, not trivial on $k$, whose valuation ring is Frobenius split. Moreover, $w$ has the feature that it is not centered on any excellent domain whose fraction field is $K$. However, even in this example, the extension $w/w^p$ is defectless. Furthermore, in light of (\ref{spherical-complete}) it is interesting to ask when a (complete) non-Archimedean and non-perfect field of prime characteristic has a Frobenius split valuation ring. Note that by completing the valuation of Example \ref{ex:non-F-split} it follows that such valuation rings are not always Frobenius split. 

There are interesting open questions pertaining to Frobenius splitting even for Noetherian valuation rings. As far as we know, it is not known if every excellent DVR of prime characteristic is Frobenius split. (\ref{noetherian-Frobenius-splitting}) (resp. (\ref{complete-DVR})) provides an affirmative answer when the fraction field of a DVR is $F$-finite (resp. when the DVR is complete). At the same time it is worth recalling that the Frobenius map is always pure for any valuation ring of prime characteristic by (\ref{flatness}), and purity seems to be a better notion than splitting for rings that are not $F$-finite. 



\bibliographystyle{amsalpha}

\end{document}